\documentclass[preprint,12pt]{elsarticle}

\usepackage{amssymb}
\usepackage{amsthm}
\usepackage{amsmath}
\usepackage{subfigure}
\usepackage{lineno}
\newtheorem{theorem}{Theorem}[section]
\newtheorem{lemma}{Lemma}[section]

\newtheorem{remark}{Remark}[section]

\newtheorem{example}{Example}[section]

\allowdisplaybreaks
\allowdisplaybreaks[4]
\begin{document}

\begin{frontmatter}

%% Title, authors and addresses

%% use the tnoteref command within \title for footnotes;
%% use the tnotetext command for theassociated footnote;
%% use the fnref command within \author or \address for footnotes;
%% use the fntext command for theassociated footnote;
%% use the corref command within \author for corresponding author footnotes;
%% use the cortext command for theassociated footnote;
%% use the ead command for the email address,
%% and the form \ead[url] for the home page:
%% \title{Title\tnoteref{label1}}
%% \tnotetext[label1]{}
%% \author{Name\corref{cor1}\fnref{label2}}
%% \ead{email address}
%% \ead[url]{home page}
%% \fntext[label2]{}
%% \cortext[cor1]{}
%% \affiliation{organization={},
%%             addressline={},
%%             city={},
%%             postcode={},
%%             state={},
%%             country={}}
%% \fntext[label3]{}

\title{Finite difference method for inhomogeneous fractional Dirichlet problem}

%% use optional labels to link authors explicitly to addresses:
%% \author[label1,label2]{}
%% \affiliation[label1]{organization={},
%%             addressline={},
%%             city={},
%%             postcode={},
%%             state={},
%%             country={}}
%%
%% \affiliation[label2]{organization={},
%%             addressline={},
%%             city={},
%%             postcode={},
%%             state={},
%%             country={}}

\author[a1]{Jing Sun}
\ead{sunj18@lzu.edu.cn}
\author[a1]{Weihua Deng}
\ead{dengwh@lzu.edu.cn}
\author[a1]{Daxin Nie}
\ead{ndx1993@163.com}

\affiliation[a1]{organization={School of Mathematics and Statistics, Gansu Key Laboratory of Applied Mathematics and Complex Systems, Lanzhou University},%Department and Organization
            city={Lanzhou},
            postcode={730000},
            country={P.R. China}}

\begin{abstract}

We make the split of the integral fractional Laplacian as  $(-\Delta)^s u=(-\Delta)(-\Delta)^{s-1}u$, where $s\in(0,\frac{1}{2})\cup(\frac{1}{2},1)$. Based on this splitting, we respectively discretize the one- and two-dimensional integral fractional Laplacian with the inhomogeneous Dirichlet boundary condition and give the corresponding truncation errors with the help of the interpolation estimate. Moreover, the  suitable corrections are proposed to guarantee the convergence in solving the inhomogeneous fractional Dirichlet problem and an $\mathcal{O}(h^{1+\alpha-2s})$ convergence rate is obtained when the solution $u\in C^{1,\alpha}(\bar{\Omega}^{\delta}_{n})$, where $n$ is the dimension of the space, $\alpha\in(\max(0,2s-1),1]$, $\delta$ is a fixed positive constant, and $h$ denotes  mesh size. Finally, the performed numerical experiments confirm the theoretical results.

%In this paper, we first provide a new decomposition for the integral fractional Laplacian according to its Fourier transform form, i.e., $(-\Delta)^s u=(-\Delta)(-\Delta)^{s-1}u$, where $s\in(0,\frac{1}{2})\cup(\frac{1}{2},1)$. Based on this decomposition form, we discretize the one- and two-dimensional integral fractional Laplacian  with the inhomogeneous Dirichlet boundary condition and give the corresponding truncation errors with the help of the interpolation estimation, respectively. Moreover, the  suitable corrections are proposed to guarantee the convergence for solving the inhomogeneous fractional Dirichlet problem and an $\mathcal{O}(h^{1+\alpha-2s})$ convergence rate can be obtained when the solution $u\in C^{1,\alpha}(\bar{\Omega}^{\delta}_{n})$, where $n$ is the dimension of the space, $\alpha\in(\max(0,2s-1),1]$, $\delta$ is a fixed positive constant, and $h$ denotes  mesh size. Finally, numerical experiments agree with the theoretical analysis.
\end{abstract}

%%Graphical abstract
%\begin{graphicalabstract}
%%\includegraphics{grabs}
%\end{graphicalabstract}

%%Research highlights
%\begin{highlights}
%\item Research highlight 1
%
%Fractional Laplacian is of wide interest to both pure and applied mathematicians, and also has extensive applications in physical and engineering community.
%
%\item Research highlight 2
%
%A fundamentally new idea of discretizing the fractional Laplacian is introduced, and the effectiveness of the designed scheme is ensured by the completely theoretical analyses and verified by numerical experiments.
%
%
%\item Research highlight 3
%
%Specific applications for simulating the mean exit time of L\'evy processes under harmonic potential are provided;  the effects of the strengthes of the potential and the L\'evy exponents are uncovered.
%
%
%
%\end{highlights}

\begin{keyword}
one- and two-dimensional integral fractional Laplacian\sep Lagrange interpolation \sep operator splitting  \sep finite difference\sep the inhomogeneous fractional Dirichlet problem \sep error estimates
\end{keyword}

\end{frontmatter}

%% \linenumbers

%% main text
%\linenumbers
\section{Introduction}
Fractional Laplacian is of wide interest to both pure and applied mathematicians, and also has extensive applications in physical and engineering community \cite{Lischke2020, Deng.2018BPftFaTFO}. Based on the splitting of the integral fractional Laplacian, we provide the finite difference approximations for the one- and two-dimensional cases of the operator. Then the approximations are used to numerically solve the inhomogeneous fractional Dirichlet problem, i.e.,
% In this work, we first provide the new finite difference approximations for one- and two-dimensional integral fractional Laplacian, respectively, and then solve the corresponding inhomogeneous fractional Dirichlet problem, i.e.
\begin{equation}\label{eqtosol}
	\left\{
	\begin{aligned}
		&(-\Delta)^{s}u(\mathbf{x})=f(\mathbf{x})\quad {\rm in}~ \Omega_{n},\\
		&u(\mathbf{x})=g(\mathbf{x})\qquad {\rm in}~ \Omega_{n}^{c},\\
	\end{aligned}
	\right .
\end{equation}
where $\Omega_{n}\subset\mathbb{R}^{n}$ $(n=1,2)$ is a bounded domain and $\Omega^{c}_{n}=\mathbb{R}^{n}\backslash\Omega_{n}$ denotes the complement of $\Omega_{n}$; $g(\mathbf{x})=0$ in $\Omega_{n}$, $g(\mathbf{x})\in L^{\infty}(\mathbb{R}^{n})$, and ${\bf supp}~g(\mathbf{x})$ is bounded;  $(-\Delta)^{s}u(\mathbf{x})$ is the integral fractional Laplacian, which can be defined by \cite{Deng.2018BPftFaTFO,acosta2017-1}
\begin{equation}\label{eqdeflap}
	(-\Delta)^{s}u(\mathbf{x})=c_{n,s} {\rm P.V.}\int_{\mathbb{R}^{n}}\frac{u(\mathbf{x})-u(\mathbf{y})}{|\mathbf{x}-\mathbf{y}|^{n+2s}}d\mathbf{y}
\end{equation}
with $c_{n,s}=\frac{2^{2s}s\Gamma(n/2+s)}{\pi^{n/2}\Gamma(1-s)}$, and  $s\in(0,\frac{1}{2})\cup(\frac{1}{2},1)$.
Moreover, \cite{Deng.2018BPftFaTFO,acosta2017-1} show that \eqref{eqdeflap} is equivalent to the following definition given via the pseudodifferential operator over the entire space $\mathbb{R}^{n}$, i.e.,
\begin{equation}\label{fourierlaplace}
	(-\Delta)^{s}u(\boldsymbol{\xi})=\mathcal{F}^{-1}(|\boldsymbol{\xi}|^{2s}\mathcal{F}(u)), \quad s>0,
\end{equation}
where $\mathcal{F}$ and $\mathcal{F}^{-1}$ stand for the Fourier transform and the inverse Fourier transform.

L\'{e}vy process is one of the most commonly used models for describing anomalous diffusion phenomena \cite{Applebaum.2009Lpasc,Sato.1999Lpaidd}, especially $\alpha$-stable L\'{e}vy process. Fractional Laplacian is introduced as the infinitesimal generator of $\alpha$-stable L\'{e}vy process \cite{Deng.2018BPftFaTFO,Gao.2014METaEPfDSDbLN}. Since the singularity and non-locality, numerical approximation of fractional Laplacian is still a challenging topic. In the past few decades, finite difference method has been widely used to approximate fractional derivatives \cite{Gao.2014METaEPfDSDbLN,Alikhanov.2015Andsfttfde,Chen.2014FOASftSFDE,duo_novel_2018,duo_2019,huang2014-1,jin_two_2016-1,jin_correction_2017,Li.2018AoLGFfTFNPP,lin_finite_2007,Nie.2019NAotTdFKEfRaDP,tian_class_2015,zhang2019}. Among them, \cite{jin_two_2016-1,jin_correction_2017,Li.2018AoLGFfTFNPP,lin_finite_2007} discretize time fractional Caputo derivative by $L_{1}$ method and convolution quadrature method; \cite{Chen.2014FOASftSFDE,tian_class_2015} provide weighted and shifted Gr\"{u}nwald difference method to discretize fractional Riesz derivative; as for fractional Laplacian, \cite{Gao.2014METaEPfDSDbLN,duo_novel_2018,duo_2019,huang2014-1} propose the finite difference scheme for solving $d$-dimensional ($d=1,2,3$) fractional Laplace equation with homogeneous Dirichlet boundary condition; moreover, the finite difference schemes provided in \cite{Nie.2019NAotTdFKEfRaDP,zhang2019} for tempered fractional Laplacian with $\lambda=0$ still apply to fractional Laplacian.

Different from the previous finite difference scheme for fractional Laplacian, we split it into the product of $(-\Delta)$ and $(-\Delta)^{s-1}$ according to its Fourier transform form, where $-\Delta$ denotes the classical Laplace operator, and $(-\Delta)^{s-1}$ (the exponent $s-1<0$) is a non-local operator without hyper-singularity (for the detailed definition, see \eqref{eqwithoutsi}). Then we use the Lagrange interpolation to discretize $(-\Delta)^{s-1}$ and the finite difference to $-\Delta$ for one- and two-dimensional cases, respectively. Moreover, some corrections are made to ensure the convergence when using our discretization to solve Eq. \eqref{eqtosol}.
 Compared with the discretizations in \cite{duo_novel_2018,duo_2019}, our scheme can deal with the inhomogeneous fractional Dirichlet problem more easily and accurately. Different from the discretizations proposed in \cite{Nie.2019NAotTdFKEfRaDP,zhang2019},  the current discretization can produce a Toeplitz matrix in one-dimensional case and a block-Toeplitz-Toeplitz-block for two-dimensional case; so fast Fourier Transform can be directly used to speed up the evaluation \cite{Chen.2005Mptaa}. Besides, we use  some examples to verify the effectiveness of the designed scheme, including truncation errors, convergence, and the simulation of the mean exit time of L\'{e}vy motion with generator $\mathcal{A}=\nabla P(x)\cdot\nabla+ (-\Delta)^{s}$; the detailed results can refer to Section 5.

The rest of the paper is organized as follows. In Section 2, we discretize one- and two-dimensional fractional Laplacian by using the Lagrange interpolation and the finite difference method. In Section 3, we provide the truncation errors for one- and two-dimensional cases, respectively. In Section 4, we make some corrections to ensure the convergence in solving the inhomogeneous fractional Dirichlet problem. Section 5 provides some numerical experiments to validate the effectiveness of the designed scheme. We conclude the paper with some discussions in the last section. Throughout the paper, $C$ is a positive constant and may be different at each occurrence.
\section{Numerical discretization of the one- and two-dimensional integral fractional Laplacian}
In this section, we first introduce a new presentation of integral fractional Laplacian according to its Fourier transform form,  and then the detailed discretizations of one- and two-dimensional integral fractional Laplacian based on the Lagrange interpolation and finite difference method are provided.

From \eqref{fourierlaplace}, one can split the fractional Laplacian in frequency domain into
\begin{equation}\label{fourieq}
	\mathcal{F}((-\Delta)^{s}u)(\boldsymbol{\xi})=|\boldsymbol{\xi}|^{2}|\boldsymbol{\xi}|^{2s-2}\mathcal{F}(u).
\end{equation}
So for $s\in(0,\frac{1}{2})\cup(\frac{1}{2},1)$, we get a new presentation of fractional Laplacian  after recovering \eqref{fourieq} to the corresponding time domain, i.e.,
\begin{equation}\label{eqreprefl}
	(-\Delta)^{s}u=(-\Delta)(-\Delta)^{s-1}u,
\end{equation}
where $(-\Delta)$ denotes the classical Laplace operator and $(-\Delta)^{s-1}$ is defined as \cite{Vazquez2012}
\begin{equation}\label{eqwithoutsi}
	(-\Delta)^{s'}u=c_{n,s'}\int_{\mathbb{R}^{n}}|\mathbf{x}-\mathbf{y}|^{-2s'-n}u(\mathbf{y})d\mathbf{y},\qquad s'<0
\end{equation}
with $c_{n,s'}=-\frac{2^{2s'}s'\Gamma(n/2+s')}{\pi^{n/2}\Gamma(1-s')}$ for $s'<0$.

Below, we provide the detailed discretization for one- and two-dimensional fractional Laplacian based on the splitting \eqref{eqreprefl}, respectively.
\subsection{One-dimensional discretization}
Here we focus on the discretization of $(-\Delta)^{s}u$ with the inhomogeneous Dirichlet boundary condition in one-dimensional case.
Suppose the bounded domain $\Omega_{1}=[-L,L]$ and $u=g(x)$ in $\Omega_{1}^{c}$; set $h=2L/N$ with $N\in\mathbb{N}$ and $x_{i}=-L+ih$, $i\in\mathbb{Z}$. Introduce $I_{i}=[x_{i-1},x_{i+1}]\cap\Omega_{1}$,  $i=0,1,2,\ldots,N$. Denote $\phi_{i}(x)$ as the Lagrange basis polynomial on $I_{i}$, i.e.,
\begin{equation}\label{eqdefphii}
	\phi_{i}(x)=\bar{\phi}_{1}(x-x_{i})\chi_{I_{i}}(x),
\end{equation}
where $\chi_{I_{i}}(x)$ is the characteristic function on $I_{i}$ and $\bar{\phi}_{1}(y)$ is defined by
\begin{equation*}
	\bar{\phi}_{1}(y)=\left\{\begin{aligned}
		1-\frac{|y|}{h},&\quad y\in (-h,h),\\
		0,&\quad y\notin (-h,h).
	\end{aligned}\right.
\end{equation*}
Thus $u(x)$ can be approximated by
\begin{equation*}
	u(x)\approx\mathbb{I}_{1}u(x)=\sum_{i=0}^{N}u_{i}\phi_{i}(x)+g(x),
\end{equation*}
where $u_{i}=u(x_{i})$ and $\mathbb{I}_{1}$ means the interpolation operator here. So we can approximate $(-\Delta)^{s-1}u$ by using

\begin{equation*}
	\begin{aligned}
		(-\Delta)_{h}^{s-1}u(x_{i})=& c_{1,s-1}\int_{\Omega_{1}}|x_{i}-y|^{1-2s}\sum_{j=0}^{N}u_{j}\phi_{j}(y)dy\\
		&+c_{1,s-1}\int_{\mathbb{R}}|x_{i}-y|^{1-2s}g(y)dy
		%=& c_{1,s-1}\sum_{j=1}^{N-1}\int_{I_{j}}|x_{i}-y|^{1-2s}\phi_{j}(y)dyu_{j}+R_{i}\\
		=\sum_{j=1}^{N-1}\bar{\omega}_{j-i}u_{j}+R_{i},
	\end{aligned}
\end{equation*}
where, for $0< i,j< N$,
\begin{equation}\label{eqdefomgk}
	\begin{aligned}
		\bar{\omega}_{j-i}=&c_{1,s-1}\int_{I_{j}}|x_{i}-y|^{1-2s}\phi_{j}(y)dy
		=c_{1,s-1}\int_{-h}^{h}|(j-i)h-y|^{1-2s}\bar{\phi}_{1}(y)dy
	\end{aligned}
\end{equation}
and
\begin{equation*}
	R_{i}=c_{1,s-1}\int_{\mathbb{R}}|x_{i}-y|^{1-2s}(u_{0}\phi_{0}(y)+u_{N}\phi_{N}(y)+g(y))dy.
\end{equation*}
As for $(-\Delta)$, we can approximate it by
\begin{equation*}
	(-\Delta)u_i\approx(-\Delta)_{h}u_i=-\frac{u_{i-1}-2u_{i}+u_{i+1}}{h^{2}}.
\end{equation*}
According to \eqref{eqreprefl}, we obtain the approximation of fractional Laplacian $(-\Delta)^{s}u$ with  $s\in(0,\frac{1}{2})\cup(\frac{1}{2},1)$, i.e.,
\begin{equation}\label{eqdefFLH1D}
	\begin{aligned}
		(-\Delta)^{s}u_i\approx(-\Delta)^{s}_{h}u_i=&-\frac{(-\Delta)^{s-1}_{h}u_{i-1}-2(-\Delta)^{s-1}_{h}u_{i}+(-\Delta)^{s-1}_{h}u_{i+1}}{h^{2}}\\
		=&\sum_{j=1}^{N-1}w_{j-i}u_{j}+(-\Delta)_{h}R_{i},%\frac{R_{i-1}-2R_{i}+R_{i+1}}{h^{2}},
	\end{aligned}
\end{equation}
where
\begin{equation}\label{eqdefw}
	w_{i}=(-\Delta)_{h}\bar{\omega}_{i}.%-\frac{\bar{\omega}_{i-1}-2\bar{\omega}_{i}+\bar{\omega}_{i+1}}{h^{2}}.
\end{equation}

\subsection{Two-dimensional discretization}
Here we discretize $(-\Delta)^{s}u$ with the inhomogeneous Dirichlet boundary condition in two-dimensional case. Suppose the bounded domain $\Omega_{2}=[-L,L]\times[-L,L]$, $u=g(x,y)$ in $\Omega_{2}^{c}$, the mesh size $h=2L/N$, $N\in\mathbb{N}$, and $(x_{i},y_{j})=(-L+ih,-L+jh)$, $i,~j\in \mathbb{Z}$. Denote $\phi_{i,j}$ as the Lagrange basis polynomial on $I_{i,j}=[x_{i-1},x_{i+1}]\times[y_{j-1},y_{j+1}]\cap\Omega_{2}$, $i,~j=0,1,2,\ldots,N$, i.e.,
\begin{equation}\label{eqdefphii2D}
	\phi_{i,j}(x,y)=\bar{\phi}_{2}(x-x_{i},y-y_{j})\chi_{I_{i,j}}(x,y),
%	\left\{
%	\begin{aligned}
%		&\left (1-\frac{|x-x_i|}{h}\right )\left (1-\frac{|y-y_j|}{h}\right ),\qquad x\in I_{i,j},\\
%		&0,\quad\qquad\qquad x\notin I_{i,j}.
%	\end{aligned}\right.
\end{equation}
where $\chi_{I_{i,j}}(x,y)$ is the characteristic function on $I_{i,j}$ and $\bar{\phi}_{2}(x,y)$ is defined by
\begin{equation*}
	\bar{\phi}_{2}(x,y)=\left\{\begin{aligned}
		\left (1-\frac{|x|}{h}\right )\left (1-\frac{|y|}{h}\right ),&\quad (x,y)\in (-h,h)\times(-h,h),\\
		0,&\quad (x,y)\notin (-h,h)\times(-h,h).
	\end{aligned}\right.
\end{equation*}
Introducing $\mathbb{I}_{2}$ as the interpolation operator in two space dimensions, one has
\begin{equation*}
	u\approx\mathbb{I}_{2}u=\sum_{i=0}^{N}\sum_{j=0}^{N}u_{i,j}\phi_{i,j}+g(x,y),
\end{equation*}
where $u_{i,j}=u(x_{i},y_{j})$.   Similarly, $(-\Delta)^{s-1}u(x,y)$ can be approximated by
\begin{equation*}
	\begin{aligned}
		&(-\Delta)^{s-1}_{h}u(x_{i},y_{j})
		%		\\
		%		=& c_{2,s-1}\int_{-L}^{L}\int_{-L}^{L}|(x_{i},y_{j})-(\xi,\eta)|^{-2s}\sum_{p=0}^{N}\sum_{q=0}^{N}u_{p,q}\phi_{p,q}(\xi,\eta)d\xi d\eta\\
		%&+c_{2,s-1}\int\int_{\mathbb{R}^2}|(x_{i},y_{j})-(\xi,\eta)|^{-2s}g(\xi,\eta)d\xi d\eta\\
		%		=& c_{2,s-1}\sum_{p=1}^{N-1}\sum_{q=1}^{N-1}\int\int_{I_{p,q}}|(x_{i},y_{j})-(\xi,\eta)|^{-2s}\phi_{p,q}(\xi,\eta)d\xi d\eta u_{p,q}+R_{i,j}\\
		=&\sum_{p=1}^{N-1}\sum_{q=1}^{N-1}\bar{\omega}_{p-i,q-j}u_{p,q}+R_{i,j},
	\end{aligned}
\end{equation*}
where $|(x_i,y_j)-(\xi,\eta)|= \sqrt{(x_i-\xi)^{2}+(y_j-\eta)^{2}}$, and for $0<i,j,p,q<N$,
\begin{equation}\label{eqdefomgk2D}
	\begin{aligned}
		\bar{\omega}_{p-i,q-j}=&c_{2,s-1}\int\int_{I_{p,q}}|(x_{i},y_{j})-(\xi,\eta)|^{-2s}\phi_{p,q}(\xi,\eta)d\xi d\eta\\
		=&c_{2,s-1}\int_{-h}^{h}\int_{-h}^{h}|((p-i)h,(q-j)h)+(\xi,\eta)|^{-2s}\bar{\phi}_{2}(\xi,\eta)d\xi d\eta,
	\end{aligned}
\end{equation}
and
\begin{equation*}
	\begin{aligned}
		R_{i,j}=&c_{2,s-1}\int\int_{\mathbb{R}^{2}}|(x_{i},y_{j})-(\xi,\eta)|^{-2s}g(\xi,\eta)d\xi d\eta\\
		&+c_{2,s-1}\sum_{pq(p-N)(q-N)=0,0\leq p,q\leq N}\int\int_{\mathbb{R}^{2}}|(x_{i},y_{j})-(\xi,\eta)|^{-2s}u_{p,q}\phi_{p,q}d\xi d\eta.	
		%		&+c_{2,s-1}\sum_{p=1}^{N-1}\int\int_{\mathbb{R}^{2}}|(x_{i},y_{j})-(\xi,\eta)|^{-2s}(u_{0,p}\phi_{0,p}+u_{N,p}\phi_{N,p}\\
		%		&+u_{p,0}\phi_{p,0}+u_{p,N}\phi_{p,N})d\xi d\eta\\
		%		&+c_{2,s-1}\int\int_{\mathbb{R}^{2}}|(x_{i},y_{j})-(\xi,\eta)|^{-2s}(u_{0,0}\phi_{0,0}+u_{N,0}\phi_{N,0}\\
		%		&+u_{0,N}\phi_{0,N}+u_{N,N}\phi_{N,N})d\xi d\eta.
	\end{aligned}
\end{equation*}
Next, using the following formula to approximate $(-\Delta)$, i.e.,
\begin{equation*}
	(-\Delta)u_{i,j}\approx(-\Delta)_{h,1}u_{i,j}=-\frac{u_{i-1,j}+u_{i+1,j}+u_{i,j+1}+u_{i,j-1}-4u_{i,j}}{h^{2}},
\end{equation*}
one can get the approximation of $(-\Delta)^{s}u$, i.e.,
\begin{equation}\label{eqdefFLH2D}
	\begin{aligned}
		(-\Delta)^{s}_{h,1}u_{i,j}
		=(-\Delta)_{h,1}(-\Delta)^{s-1}_{h}u_{i,j}
		=\sum_{i=1}^{N-1}\sum_{j=1}^{N-1}w^{(1)}
		_{p-i,q-j}u_{i,j}+(-\Delta)_{h,1}R_{i,j},\\
		%&-\frac{R_{i-1,j}+R_{i+1,j}+R_{i,j-1}+R_{i,j+1}-4R_{i,j}}{h^{2}}
	\end{aligned}
\end{equation}
where
\begin{equation}\label{eqdefw2D}
	w^{(1)}_{i,j}=(-\Delta)_{h,1}\bar{\omega}_{i,j}.%-\frac{\bar{\omega}_{i-1,j}+\bar{\omega}_{i+1,j}+\bar{\omega}_{i,j-1}+\bar{\omega}_{i,j+1}-4\bar{\omega}_{i,j}}{h^{2}}.
\end{equation}
An alternative approximation for $(-\Delta)u$ can be got by  using following formula, i.e.,
\begin{equation*}
	(-\Delta)u_{i,j}\approx(-\Delta)_{h,2}u_{i,j}=-\frac{u_{i-1,j-1}+u_{i+1,j-1}+u_{i-1,j+1}+u_{i+1,j+1}-4u_{i,j}}{2h^{2}}.
\end{equation*}
Also, $(-\Delta)^{s}u$ can be discretized as
\begin{equation}\label{eqdefFLH2D_2}
	\begin{aligned}
		(-\Delta)^{s}_{h,2}u_{i,j}
		=(-\Delta)_{h,2}(-\Delta)^{s-1}_{h}u_{i,j}
		=&\sum_{i=1}^{N-1}\sum_{j=1}^{N-1}w^{(2)}
		_{p-i,q-j}u_{i,j}+(-\Delta)_{h,2}R_{i,j},\\
		%&-\frac{R_{i-1,j-1}+R_{i+1,j-1}+R_{i-1,j+1}+R_{i+1,j+1}-4R_{i,j}}{2h^{2}},
	\end{aligned}
\end{equation}
where
\begin{equation}\label{eqdefw2D_2}
	w^{(2)}_{i,j}=(-\Delta)_{h,2}\bar{\omega}_{i,j}.%-\frac{\bar{\omega}_{i-1,j-1}+\bar{\omega}_{i+1,j-1}+\bar{\omega}_{i-1,j+1}+\bar{\omega}_{i+1,j+1}-4\bar{\omega}_{i,j}}{2h^{2}}.
\end{equation}
Thus $(-\Delta)^{s}$ with  $s\in(0,\frac{1}{2})\cup(\frac{1}{2},1)$ can be approximated by the convex combination of \eqref{eqdefFLH2D} and \eqref{eqdefFLH2D_2}, i.e.,
\begin{equation}\label{eqdefFLH2Dall}
	(-\Delta)^{s}u\approx(-\Delta)^{s}_{h}u=\theta(-\Delta)^{s}_{h,1}u+(1-\theta)(-\Delta)^{s}_{h,2}u,\quad \theta\in[0,1],
\end{equation}
which means
\begin{equation*}
	\begin{aligned}
		(-\Delta)^{s}_{h}u_{i,j}=&\sum_{i=1}^{N-1}\sum_{j=1}^{N-1}w
		_{p-i,q-j}u_{i,j}+\left (\theta(-\Delta)^{s}_{h,1}+(1-\theta)(-\Delta)^{s}_{h,2}\right )R_{i,j},\\
		%&-\theta\frac{R_{i-1,j}+R_{i+1,j}+R_{i,j-1}+R_{i,j+1}-4R_{i,j}}{h^{2}}\\
		%&-(1-\theta)\frac{R_{i-1,j-1}+R_{i+1,j-1}+R_{i-1,j+1}+R_{i+1,j+1}-4R_{i,j}}{2h^{2}},
	\end{aligned}	
\end{equation*}
where
\begin{equation}\label{eqdefw2Dtheta}
	w_{i,j}=\theta w^{(1)}_{i,j}+(1-\theta)w^{(2)}_{i,j}.
\end{equation}
Here $w^{(1)}_{i,j}$ and $w^{(2)}_{i,j}$ are defined in \eqref{eqdefw2D} and \eqref{eqdefw2D_2}, respectively.
\section{Truncation errors}
In this section, we provide the estimate of $\|(-\Delta)^{s}u-(-\Delta)^{s}_{h}u\|_{\infty}$ in one- and two-dimensional cases, respectively. In the following, we denote $\|\cdot\|_{\infty}$ and $\|\cdot\|_{2}$ as the discrete $l^{\infty}$ and $l^{2}$ norms, and $\|\cdot\|_{L^{\infty}}$ as continuous ${L^{\infty}}$ norm.
\begin{theorem}\label{onetwodimendis}
	Let $s\in(0,\frac{1}{2})\cup(\frac{1}{2},1)$. Suppose $(-\Delta)^{s}$ and $(-\Delta)^{s}_{h}$ are defined in \eqref{eqdeflap} and \eqref{eqdefFLH1D} or \eqref{eqdefFLH2Dall}, respectively. If $u\in C^{1,\alpha}(\bar{\Omega}^{\delta}_{n})$ with some fixed constant $\delta> 4h>0$ and $\alpha\in(\max(0,2s-1),1]$, then we have
	\begin{equation*}
		\|((-\Delta)^{s}-(-\Delta)^{s}_{h})u\|_{\infty}\leq Ch^{1+\alpha-2s},\quad 	\|((-\Delta)^{s}-(-\Delta)^{s}_{h})u\|_{2}\leq Ch^{1+\alpha-2s},
	\end{equation*}
	where  $\Omega^{\delta}_{n}=((-L-\delta,L+\delta))^{n}$, $n=1,2$.
\end{theorem}
Here, we only provide the proof in two-dimensional case in detail; and the proof in one-dimensional case can be got similarly.
\begin{proof}[Proof of Theorem \ref{onetwodimendis} in two dimensions] For fixed $i,j$, according to \eqref{eqdefFLH2Dall}, we have
	\begin{equation}\label{eqtrunall}
		\begin{aligned}
			|((-\Delta)^{s}-(-\Delta)^{s}_{h})u_{i,j}|\leq &\theta|((-\Delta)^{s}-(-\Delta)^{s}_{h,1})u_{i,j}|\\
			&+(1-\theta)|((-\Delta)^{s}-(-\Delta)^{s}_{h,2})u_{i,j}|,~~\theta\in[0,1].
		\end{aligned}
	\end{equation}
Using the definitions of  $(-\Delta)^{s}$ and $(-\Delta)^{s}_{h,1}$ results in
	\begin{equation*}
		\begin{aligned}
			&|((-\Delta)^{s}-(-\Delta)_{h,1}^{s})u_{i,j}|\\
			\leq &|((-\Delta)(-\Delta)^{s-1}-(-\Delta)_{h,1}(-\Delta)^{s-1})u_{i,j}|\\
			&+|((-\Delta)_{h,1}(-\Delta)^{s-1}-(-\Delta)_{h,1}(-\Delta)_{h}^{s-1})u_{i,j}|\\
			\leq &\uppercase\expandafter{\romannumeral1}+\uppercase\expandafter{\romannumeral2}.
		\end{aligned}
	\end{equation*}
	Let $\Phi(x_{i}-\xi,y_{j}-\eta)\in C^{2}_{0}(\Omega_{2}^{\delta-h})$, which satisfies $\Phi(x_{i}-\xi,y_{j}-\eta)=((x_{i}-\xi)^{2}+(y_{j}-\eta)^{2})^{-s}$ if $(\xi,\eta)\in\Omega^{\delta/2+h}_{2}\backslash(x_{i}-h,x_{i}+h)\times(y_{j}-h,y_{j}+h)$, and
	\begin{equation*}
		\begin{aligned}
			&\left \|\Phi(x,y)\right \|_{L^{\infty}(\mathbb{R}^{2})}\leq Ch^{-2s}; \quad \left \|\frac{\partial\Phi(x,y)}{\partial x}\right \|_{L^{\infty}(\mathbb{R}^{2})},\left \|\frac{\partial \Phi(x,y)}{\partial y}\right \|_{L^{\infty}(\mathbb{R}^{2})}\leq Ch^{-1-2s};\\
			&\left \|\frac{\partial^{2}\Phi(x,y)}{\partial x^{2}}\right \|_{L^{\infty}(\mathbb{R}^{2})},\left \|\frac{\partial^{2} \Phi(x,y)}{\partial y^{2}}\right \|_{L^{\infty}(\mathbb{R}^{2})}\leq Ch^{-2-2s};\\
			&\left \|\frac{\partial^{4}\Phi(x,y)}{\partial x^{4}}\right \|_{L^{\infty}(\mathbb{R}^{2}\backslash\Omega_{2})},\left \|\frac{\partial^{4} \Phi(x,y)}{\partial y^{4}}\right \|_{L^{\infty}(\mathbb{R}^{2}\backslash\Omega_{2})}\leq C.\\
		\end{aligned}
	\end{equation*}
	Introduce the notations
	\begin{equation*}
		\frac{\partial^{2} \mu^{x}}{\partial x^{2}}=\frac{\partial^{2} \mu^{y}}{\partial y^{2}}=u.
	\end{equation*}
%	Integration by parts gives
%	\begin{equation*}
%		\begin{aligned}
%			&\int_{-L-\delta}^{L+\delta}\int_{-L-\delta}^{L+\delta}\Phi(x_{i}-\xi,y_{j}-\eta)u(\xi,\eta)d\xi d\eta\\
%			&\qquad=\int_{-L-\delta}^{L+\delta}\int_{-L-\delta}^{L+\delta}\frac{\partial^{2} \Phi(x_{i}-\xi,y_{j}-\eta)}{\partial \xi^{2}} \mu^{x}(\xi,\eta)d\xi d\eta
%		\end{aligned}
%	\end{equation*}
%	and
%	\begin{equation*}
%		\begin{aligned}
%			&\int_{-L-\delta}^{L+\delta}\int_{-L-\delta}^{L+\delta}\Phi(x_{i}-\xi,y_{j}-\eta)u(\xi,\eta)d\xi d\eta\\
%			&\qquad=\int_{-L-\delta}^{L+\delta}\int_{-L-\delta}^{L+\delta}\frac{\partial^{2} \Phi(x_{i}-\xi,y_{j}-\eta)}{\partial \eta^{2}} \mu^{y}(\xi,\eta)d\xi d\eta.
%		\end{aligned}
%	\end{equation*}
	Here, divide $\uppercase\expandafter{\romannumeral1}$  into two parts, i.e.,
	\begin{equation*}
		\begin{aligned}
			\uppercase\expandafter{\romannumeral1}\leq& C\Bigg |((-\Delta)_{x}-(-\Delta)_{x,h,1})\int\int_{\mathbb{R}^{2}}|(x_{i},y_{j})-(\xi,\eta)|^{-2s}u(\xi,\eta)d\xi d\eta\Bigg |\\
			&+C\Bigg |((-\Delta)_{y}-(-\Delta)_{y,h,1})\int\int_{\mathbb{R}^{2}}|(x_{i},y_{j})-(\xi,\eta)|^{-2s}u(\xi,\eta)d\xi d\eta\Bigg |\\
			\leq& \uppercase\expandafter{\romannumeral1}^{x}+\uppercase\expandafter{\romannumeral1}^{y},
		\end{aligned}
	\end{equation*}
	where $(-\Delta)_{x}=-\frac{\partial^{2}}{\partial x^{2}}$, $(-\Delta)_{y}=-\frac{\partial^{2}}{\partial y^{2}}$ and
	\begin{equation*}
		\begin{aligned}
			&(-\Delta)_{x,h,1}v_{i,j}=-\frac{v_{i-1,j}-2v_{i,j}+v_{i+1,j}}{h^{2}},\ (-\Delta)_{y,h,1}v_{i,j}=-\frac{v_{i,j-1}-2v_{i,j}+v_{i,j+1}}{h^{2}}.\\
		\end{aligned}
	\end{equation*}
	Introduce $\Psi(x_{i}-\xi,y_{j}-\eta)=|(x_{i},y_{j})-(\xi,\eta)|^{-2s}-\Phi(x_{i}-\xi,y_{j}-\eta)$. For $\uppercase\expandafter{\romannumeral1}^{x}$, we find
	\begin{equation*}
		\begin{aligned}
			&\uppercase\expandafter{\romannumeral1}^{x}\leq
			 C\Bigg |((-\Delta)_{x}-(-\Delta)_{x,h,1})\int\int_{\mathbb{R}^{2}}\Phi(x_{i}-\xi,y_{j}-\eta)u(\xi,\eta)d\xi d\eta\Bigg |\\
			&~~+ C\Bigg |((-\Delta)_{x}-(-\Delta)_{x,h,1})\int\int_{\mathbb{R}^{2}}\Psi(x_{i}-\xi,y_{j}-\eta)u(\xi,\eta)d\xi d\eta\Bigg |\\
			&\leq \uppercase\expandafter{\romannumeral1}^{x}_{1}+\uppercase\expandafter{\romannumeral1}^{x}_{2}.
		\end{aligned}
	\end{equation*}
	Introduce $\mathbb{D}^{\delta}_{i,j}=\{(\xi,\eta)|(x_{i}-\xi,y_{j}-\eta)\in \Omega_{2}^{\delta}\}$ and $\mathbb{D}^{\delta,0}_{i,j}=\mathbb{D}^{\delta}_{i,j}\backslash(-h,h)\times(-h,h)$.
	Simple calculations imply
	\begin{equation*}
		\begin{aligned}
			\uppercase\expandafter{\romannumeral1}^{x}_{1}\leq&  C\Bigg |((-\Delta)_{x}-(-\Delta)_{x,h,1})\int\int_{\Omega_{2}^{\delta}}\frac{\partial^{2}\Phi(x_{i}-\xi,y_{j}-\eta)}{\partial \xi^{2}}\mu^{x}(\xi,\eta)d\xi d\eta\Bigg |\\
			\leq&  C\Bigg |((-\Delta)_{x}-(-\Delta)_{x,h,1})\int\int_{\mathbb{D}^{\delta}_{i,j}}\frac{\partial^{2}\Phi(\xi,\eta)}{\partial \xi^{2}}\mu^{x}(x_{i}-\xi,y_{j}-\eta)d\xi d\eta\Bigg |\\
			\leq&  C\Bigg |\int_{-h}^{h}\int_{-h}^{h}\frac{\partial^{2}\Phi(\xi,\eta)}{\partial \xi^{2}}((-\Delta)_{x}-(-\Delta)_{x,h,1})\mu^{x}(x_{i}-\xi,y_{j}-\eta)d\xi d\eta\Bigg |\\
			& +C\Bigg |\int\int_{\mathbb{D}^{\delta,0}_{i,j}}\frac{\partial^{2}\Phi(\xi,\eta)}{\partial \xi^{2}}((-\Delta)_{x}-(-\Delta)_{x,h,1})\mu^{x}(x_{i}-\xi,y_{j}-\eta)d\xi d\eta\Bigg |\\
			\leq& \uppercase\expandafter{\romannumeral1}^{x}_{1,1}+\uppercase\expandafter{\romannumeral1}^{x}_{1,2}.
		\end{aligned}
	\end{equation*}
	By Taylor's expansion, we have $|((-\Delta)_{x}-(-\Delta)_{x,h})v(x_{i})|\leq Ch^{1+\alpha}\|v\|_{C^{3,\alpha}([x_{i-1},x_{i+1}])}$ for $v\in C^{3,\alpha}([x_{i-1},x_{i+1}])$. Thus there holds
	\begin{equation*}
		\begin{aligned}
			\uppercase\expandafter{\romannumeral1}^{x}_{1,1}\leq& Ch^{1+\alpha}\int_{-h}^{h}\int_{-h}^{h}\Bigg |\frac{\partial^{2}\Phi(\xi,\eta)}{\partial \xi^{2}}\Bigg |d\xi d\eta\|u\|_{C^{1,\alpha}(\bar{\Omega}^{\delta}_{2})}\\
			\leq& Ch^{1+\alpha-2s}\|u\|_{C^{1,\alpha}(\bar{\Omega}^{\delta}_{2})}.
		\end{aligned}
	\end{equation*}
	Using the fact
	\begin{equation*}
		\begin{aligned}
			&\left|\frac{\partial^{2}|(x,y)-(\xi,\eta)|^{-2s}}{\partial \xi^{2}}\right|\leq C |(x,y)-(\xi,\eta)|^{-2s-2}\quad {\rm for}~~(x,y)\neq(\xi,\eta),\\
%			\leq &C |(x,y)-(\xi,\eta)|^{-2s-2}+C(x_{i}-\xi)^{2} |(x,y)-(\xi,\eta)|^{-2s-4}\\
%			\leq&C |(x,y)-(\xi,\eta)|^{-2s-2},
		\end{aligned}
	\end{equation*}
	we obtain
	\begin{equation*}
		\begin{aligned}
			\uppercase\expandafter{\romannumeral1}^{x}_{1,2}\leq& Ch^{1+\alpha}\int\int_{\mathbb{D}^{\delta,0}_{i,j}}\Bigg |\frac{\partial^{2}|(x_{i},y_{j})-(\xi,\eta)|}{\partial \xi^{2}}\Bigg |d\xi d\eta\|u\|_{C^{1,\alpha}(\bar{\Omega}^{\delta}_{2})}\\
			\leq& Ch^{1+\alpha-2s}\|u\|_{C^{1,\alpha}(\bar{\Omega}^{\delta}_{2})}.
		\end{aligned}
	\end{equation*}
	Decomposing $\uppercase\expandafter{\romannumeral1}^{x}_{2}$ into three parts leads to
	\begin{align*}
			&\uppercase\expandafter{\romannumeral1}^{x}_{2}\leq	 C\Bigg |(-\Delta)_{x}\int_{x_{i}-h}^{x_{i}+h}\int_{y_{j}-h}^{y_{j}+h}\Psi(x_{i}-\xi,y_{j}-\eta)u(\xi,\eta)d\xi d\eta\Bigg |\\
			&~~+C\Bigg |(-\Delta)_{x,h,1}\int_{x_{i}-h}^{x_{i}+h}\int_{y_{j}-h}^{y_{j}+h}\Psi(x_{i}-\xi,y_{j}-\eta)u(\xi,\eta)d\xi d\eta\Bigg |\\
			&~~+C\Bigg |((-\Delta)_{x}-(-\Delta)_{x,h,1})\int\int_{\Omega_{2}^{c}}\Psi(x_{i}-\xi,y_{j}-\eta)u(\xi,\eta)d\xi d\eta\Bigg |\\
			&~~\leq \uppercase\expandafter{\romannumeral1}^{x}_{2,1}+\uppercase\expandafter{\romannumeral1}^{x}_{2,2}+\uppercase\expandafter{\romannumeral1}^{x}_{2,3}.
	\end{align*}
	For $\uppercase\expandafter{\romannumeral1}^{x}_{2,1}$, we get, for some function $C_{0}(y)$ independent of $x$,
	\begin{equation*}
		\begin{aligned}
			\uppercase\expandafter{\romannumeral1}^{x}_{2,1}\leq &C\Bigg |\frac{\partial}{\partial x}\int_{x_{i}-h}^{x_{i}+h}\int_{y_{j}-h}^{y_{j}+h}\Psi(x_{i}-\xi,y_{j}-\eta)\frac{\partial u(\xi,\eta)}{\partial \xi}d\xi d\eta\Bigg |\\
			\leq &C\Bigg |\frac{\partial}{\partial x}\int_{x_{i}-h}^{x_{i}+h}\int_{y_{j}-h}^{y_{j}+h}\Psi(x_{i}-\xi,y_{j}-\eta)\left (\frac{\partial u(\xi,\eta)}{\partial \xi}-C_{0}(\eta)\right )d\xi d\eta\Bigg |\\
			\leq &C\Bigg |\int_{x_{i}-h}^{x_{i}+h}\int_{y_{j}-h}^{y_{j}+h}\frac{\partial\Psi(x_{i}-\xi,y_{j}-\eta)}{\partial x}\left (\frac{\partial u(\xi,\eta)}{\partial \xi}-C_{0}(\eta)\right )d\xi d\eta\Bigg |.\\
		\end{aligned}
	\end{equation*}
	Choosing $C_{0}(y)=\frac{\partial u(x,y)}{\partial x}|_{x=x_{i}}$ results in
	\begin{equation*}
		\uppercase\expandafter{\romannumeral1}^{x}_{2,1}\leq Ch^{1+\alpha-2s}\|u\|_{C^{1,\alpha}(\bar{\Omega}^{\delta}_{2})}.
	\end{equation*}
	By using $|\Phi(\xi,\eta)|\leq Ch^{-2s}$ and the Taylor expansion, there holds
	\begin{equation*}
		\begin{aligned}
			\uppercase\expandafter{\romannumeral1}^{x}_{2,2}\leq& C\Bigg |\int_{-h}^{h}\int_{-h}^{h}\Psi(\xi,\eta)(-\Delta)_{x,h,1}u(x_{i}-\xi,y_{j}-\eta)d\xi d\eta\Bigg |\\
			\leq &Ch^{1+\alpha-2s}\|u\|_{C^{1,\alpha}(\bar{\Omega}^{\delta}_{2})}.
		\end{aligned}
	\end{equation*}
	Simple calculations imply
	\begin{equation*}
		\begin{aligned}
			\uppercase\expandafter{\romannumeral1}^{x}_{2,3}\leq &Ch^{2}\delta^{-2-2s}\|u\|_{L^{\infty}(\mathbb{R}^{2})}.
		\end{aligned}
	\end{equation*}
	Combining above estimates, one has
	\begin{equation*}
		\uppercase\expandafter{\romannumeral1}^{x}\leq Ch^{1+\alpha-2s}.
	\end{equation*}
	Similarly, there is
	\begin{equation*}
		\uppercase\expandafter{\romannumeral1}^{y}\leq Ch^{1+\alpha-2s}.
	\end{equation*}
	As for $\uppercase\expandafter{\romannumeral2}$, the fact $\|u(\xi,\eta)-\mathbb{I}_{2}u(\xi,\eta)\|_{L^{\infty}(\Omega_{2})}\leq Ch^{1+\alpha}\|u\|_{C^{1,\alpha}(\bar{\Omega}^{\delta}_{2})}$ implies
	\begin{equation*}
		\begin{aligned}
			\uppercase\expandafter{\romannumeral2}\leq &C\left |(-\Delta)_{h,1}\int_{-L}^{L}\int_{-L}^{L}|(x_{i},y_{j})-(\xi,\eta)|^{-2s}(u(\xi,\eta)-\mathbb{I}_{2}u(\xi,\eta))d\xi d\eta\right|\\
			\leq &Ch^{1+\alpha}\int_{-L}^{L}\int_{-L}^{L}|((-\Delta)_{h,1}|(x_{i},y_{j})-(\xi,\eta)|^{-2s})|d\xi d\eta\|u\|_{C^{1,\alpha}(\bar{\Omega}^{\delta}_{2})}.
		\end{aligned}
	\end{equation*}
	Introduce  $\mathbb{D}_{i,j}=\Omega_{2}^{\delta}\backslash(x_{i}-2h,x_{i}+2h)\times(y_{i}-2h,y_{i}+2h)$. Simple calculations give
	\begin{equation*}
		\begin{aligned}
			&\int_{-L}^{L}\int_{-L}^{L}|((-\Delta)_{h,1}|(x_{i},y_{j})-(\xi,\eta)|^{-2s})|d\xi d\eta\\
			\leq&C \int_{y_{j}-2h}^{y_{j}+2h}\int_{x_{i}-2h}^{x_{i}+2h}\left |((-\Delta)_{h,1}|(x_{i},y_{j})-(\xi,\eta)|^{-2s})\right |d\xi d\eta\\
			&+C\int\int_{\mathbb{D}_{i,j}}\left |((-\Delta)_{h,1}|(x_{i},y_{j})-(\xi,\eta)|^{-2s})\right |d\xi d\eta\\
			\leq& Ch^{-2s},
		\end{aligned}
	\end{equation*}
	which leads to
	\begin{equation*}
		\uppercase\expandafter{\romannumeral2}\leq Ch^{1+\alpha-2s}\|u\|_{C^{1,\alpha}(\bar{\Omega}^{\delta}_{2})}.
	\end{equation*}
	Thus according to $\uppercase\expandafter{\romannumeral1}$ and $\uppercase\expandafter{\romannumeral2}$, we have
	\begin{equation*}
		\|((-\Delta)^{s}-(-\Delta)^{s}_{h,1})u\|_{\infty}\leq Ch^{1+\alpha-2s},\quad \|((-\Delta)^{s}-(-\Delta)^{s}_{h,1})u\|_{2}\leq Ch^{1+\alpha-2s}.
	\end{equation*}
	As for $\|((-\Delta)^{s}-(-\Delta)^{s}_{h,2})u\|_{\infty}$, by similar arguments, we can get the estimates
	\begin{equation*}
		\|((-\Delta)^{s}-(-\Delta)^{s}_{h,2})u\|_{\infty}\leq Ch^{1+\alpha-2s},\quad \|((-\Delta)^{s}-(-\Delta)^{s}_{h,2})u\|_{2}\leq Ch^{1+\alpha-2s}.
	\end{equation*}
Collecting the above estimates, the desired results are reached.
\end{proof}

\section{Convergence in solving the inhomogeneous fractional Dirichlet problem}
In this section, we first propose the sufficient conditions for getting the convergence when using the provided discretizations to solve Eq. \eqref{eqtosol}. Then we try to modify the discretizations provided in Sec. 3 according to the  corresponding conditions. Finally, we present the convergence analyses in solving Eq. \eqref{eqtosol}.
%Below we introduce some notations to rewrite the scheme in matrix-vector form.

Now, we  first provide a lemma which is useful for the convergence analyses.
\begin{lemma}[\cite{Axelsson1994}]\label{lemmaGersgorin}
	Let matrix $\mathbf{A}$ be
	\begin{equation*}
		\mathbf{A}=\left[
		\begin{matrix}
			a_{1,1}&a_{1,2}&\cdots& a_{1,N}\\
			a_{2,1}&a_{2,2}&\cdots& a_{2,N}\\
			\vdots&\vdots&\ddots&\vdots\\
			a_{N,1}&a_{N,2}&\cdots&a_{N,N}
		\end{matrix}\right].
	\end{equation*}
	Introduce the discs:  %following discs $C_{i}$ and $C'_{i}$,
	\begin{equation}
		\begin{aligned}
			&C_i=\{z\in \mathbb{C};|z-a_{i,i}|\leq\sum_{i\neq j}|a_{i,j}|\},~1\leq i\leq N,\\
			&C'_i=\{z\in \mathbb{C};|z-a_{i,i}|\leq\sum_{i\neq j}|a_{j,i}|\},~1\leq i\leq N.
		\end{aligned}
	\end{equation}
	The spectrum $\lambda(\mathbf{A})$ of $\mathbf{A}$ is enclosed in the union of $C_{i}$ and $C'_{i}$.
\end{lemma}
Below we give two theorems to state the sufficient conditions of achieving the convergence in solving Eq. \eqref{eqtosol} in one and two dimensions, respectively.
\iffalse
\begin{theorem}\label{thmcorstd}
	Let $\mathbf{U}$  be solution of Eq. \eqref{eqtosol} and $\mathbf{U}_{h}$ be the solution of
	\begin{equation}\label{eqmatrixform}
		\mathbf{B}_{1}\mathbf{U}_{h}+\mathbf{G}=\mathbf{F},
	\end{equation}
	where $\mathbf{F}=\{f(x_{i})\}_{i=1}^{N-1}=\{f_{i}\}_{i=1}^{N-1}$, $\mathbf{G}=\{G_{i}\}_{i=1}^{N-1}$ and
	\begin{equation*}
		\mathbf{B}_{1}=\left[
		\begin{matrix}
			b_{0}& b_{1}&\cdots &b_{N-2}\\
			b_{1}& b_{0}&\cdots &b_{N-3}\\
			\vdots&\vdots&\ddots&\vdots\\
			b_{N-2}& b_{N-3}&\cdots &b_{0}\\
		\end{matrix}\right ].
	\end{equation*}
	If $\mathbf{B}_{1}$ and $\mathbf{G}$ satisfy following conditions:
	\begin{enumerate}[(1)]
		\item $\|\mathbf{F}-(\mathbf{B}_{1}\mathbf{U}+\mathbf{G})\|_{\infty}\leq Ch^{k}$;
		
		\item $b_{0}>0$, $b_{i}<0$ for $i\neq 0$;
		
		\item$\inf\limits_{i\in[1,N-1]}\sum\limits_{j=1}^{N}b_{|i-j|}>C_{0}>0$,
		
	\end{enumerate}
	then we obtain
	\begin{equation*}
		\begin{aligned}
			&\|U-U_{h}\|_{\infty}<Ch^{k},\quad\|U-U_{h}\|_{2}<Ch^{k}.
		\end{aligned}
	\end{equation*}
\end{theorem}
\fi
\begin{theorem}\label{thmcorstd}
	Given two vectors $\mathbf{F}$, $\mathbf{G}$ and the matrix
	\begin{equation*}
		\mathbf{B}_{1}=\left[
		\begin{matrix}
			b_{0}& b_{1}&\cdots &b_{N-2}\\
			b_{1}& b_{0}&\cdots &b_{N-3}\\
			\vdots&\vdots&\ddots&\vdots\\
			b_{N-2}& b_{N-3}&\cdots &b_{0}\\
		\end{matrix}\right ].
	\end{equation*}
Let $\mathbf{U}_{h}$ be the solution of the linear system
	\begin{equation}\label{eqmatrixform}
		\mathbf{B}_{1}\mathbf{U}_{h}+\mathbf{G}=\mathbf{F}.
	\end{equation}
Assume $\mathbf{U}$, $\mathbf{B}_{1}$, and $\mathbf{G}$ satisfy the conditions:
	\begin{enumerate}[(1)]
		\item $\|\mathbf{F}-(\mathbf{B}_{1}\mathbf{U}+\mathbf{G})\|_{\infty}\leq Ch^{k}$, $\|\mathbf{F}-(\mathbf{B}_{1}\mathbf{U}+\mathbf{G})\|_{2}\leq Ch^{k}$;
		
		\item $b_{0}>0$, $b_{i}<0$ for $i\neq 0$;
		
		\item there exists some constant $C_{0}>0$ such that $\inf\limits_{i=1,2\ldots,N-1}\sum\limits_{j=1}^{N-1}b_{|i-j|}>C_{0}$.
		
	\end{enumerate}
Then we obtain
	\begin{equation*}
		\begin{aligned}
			&\|\mathbf{U}-\mathbf{U}_{h}\|_{\infty}<Ch^{k},\quad\|\mathbf{U}-\mathbf{U}_{h}\|_{2}<Ch^{k}.
		\end{aligned}
	\end{equation*}
\end{theorem}
\begin{proof}

	By Lemma \ref{lemmaGersgorin} and the properties of $b_{i}$, we have
	\begin{equation*}
		\lambda_{min}(\mathbf{B}_{1})>C_{0}.
	\end{equation*}
	Let $\mathbf{e}^{\mathbf{U}}=\mathbf{U}_{h}-\mathbf{U}=\{e^{\mathbf{U}}_{i}\}_{i=1}^{N-1}$, $\bar{\mathbf{F}}=\mathbf{F}-(\mathbf{B}_{1}\mathbf{U}+\mathbf{G})=\{\bar{f}_{i}\}_{i=1}^{N-1}$. Then
	\begin{equation*}
		CC_{0}\|\mathbf{e}^{\mathbf{U}}\|^{2}_{2}\leq (\mathbf{B}_{1}\mathbf{e}^{\mathbf{U}},\mathbf{e}^{\mathbf{U}})= (\mathbf{\bar{F}},\mathbf{e}^{\mathbf{U}})\leq \|\mathbf{\bar{F}}\|_{2}\|\mathbf{e}^{\mathbf{U}}\|_{2},
	\end{equation*}
	which leads to $\|\mathbf{e}^{\mathbf{U}}\|_{2}\leq CC_{0}^{-1}\|\mathbf{\bar{F}}\|_{2}$.
	Assuming $\|\mathbf{e}^{\mathbf{U}}\|_{\infty}=|e^{\mathbf{U}}_{p}|$, we have
	\begin{equation*}
		\begin{aligned}
			&e^{\mathbf{U}}_{p}(\bar{f}_{p}-CC_{0}e^{\mathbf{U}}_{p})
			=e^{\mathbf{U}}_{p}\left (\sum_{i=1}^{i=N-1}b_{|p-i|}e^{\mathbf{U}}_{i}-CC_{0}e^{\mathbf{U}}_{p}\right )\\
			&\quad=e^{\mathbf{U}}_{p}(\sum_{i=1,i\neq p}^{i=N-1}b_{|p-i|}e^{\mathbf{U}}_{i}+(b_{0}-CC_{0})e^{\mathbf{U}}_{p})
			\geq e^{\mathbf{U}}_{p}(\sum_{i=1,i\neq p                  }^{i=N-1}b_{|p-i|}(e^{\mathbf{U}}_{i}-e^{\mathbf{U}}_{p}))\geq 0,\\
		\end{aligned}
	\end{equation*}
	which yields
	\begin{equation*}
		\|\mathbf{e}^{\mathbf{U}}\|_{\infty}\leq CC_{0}^{-1}|\bar{f}_{p}|\leq CC_{0}^{-1}\|\mathbf{\bar{F}}\|_{\infty}.
	\end{equation*}
	Combining the first condition, we can get the desired results.
\end{proof}

Similarly, for the two-dimensional case, we find
\begin{theorem}\label{thmcorstd2}
%	Suppose that
%	\begin{equation*}
%		\begin{aligned}
%			&\mathbf{U}=[u_{1,1},u_{1,2}\ldots,u_{1,N-1},u_{2,1},\ldots,u_{N-1,N-1}];\\
%			&\mathbf{U}_{h}=[u^{h}_{1,1},u^{h}_{1,2}\ldots,u^{h}_{1,N-1},u^{h}_{2,1},\ldots,u^{h}_{N-1,N-1}];\\
%			&\mathbf{F}=[f(x_{1},y_{1}),f(x_{1},y_{2})\ldots,f(x_{1},y_{N-1}),f(x_{2},y_{1}),\ldots,f(x_{N-1},y_{N-1})],\\
%		\end{aligned}
%	\end{equation*}
Suppose $\mathbf{U}$, $\mathbf{U}_{h}$, $\mathbf{G}$, and $\mathbf{F}$ satisfy
	\begin{equation}\label{eqmatrixform2}
		\mathbf{B}_{2}\mathbf{U}_{h}+\mathbf{G}=\mathbf{F},
	\end{equation}
	and
	\begin{equation*}
		\|\mathbf{F}-(\mathbf{B}_{2}\mathbf{U}+\mathbf{G})\|_{\infty}\leq Ch^{k},\quad \|\mathbf{F}-(\mathbf{B}_{2}\mathbf{U}+\mathbf{G})\|_{2}\leq Ch^{k}.
	\end{equation*}
	Here
	\begin{equation*}
		\mathbf{B}_{2}=\left[
		\begin{matrix}
			\mathbf{T}_{0}& \mathbf{T}_{1}&\cdots &\mathbf{T}_{N-1}\\
			\mathbf{T}_{-1}& \mathbf{T}_{0}&\cdots &\mathbf{T}_{N-2}\\
			\vdots&\vdots&\ddots&\vdots\\
			\mathbf{T}_{-N+1}& \mathbf{T}_{-N+2}&\cdots &\mathbf{T}_{0}
		\end{matrix}
		\right ],~ \mathbf{T}_{k}=\left[
		\begin{matrix}
			t_{k,0}& t_{k,1}&\cdots &t_{k,N-2}\\
			t_{k,1}& t_{k,0}&\cdots &t_{k,N-3}\\
			\vdots&\vdots&\ddots&\vdots\\
			t_{k,N-2}& t_{k,N-3}&\cdots &t_{k,0}
		\end{matrix}
		\right ].
	\end{equation*}
Assume the following conditions are satisfied,
	\begin{enumerate}[(1)]
		\item $t_{k,i}>0$ for $k=i=0$, otherwise, $t_{k,i}<0$;
		\item $\inf_{p,q=1,\ldots,N-1}\sum_{i,j=1}^{N-1}t_{|p-i|,|q-j|}>C_{0}>0$.
	\end{enumerate}
Then one has
	\begin{equation*}
		\begin{aligned}
			&\|\mathbf{U}-\mathbf{U}_{h}\|_{\infty}<Ch^{k},\quad\|\mathbf{U}-\mathbf{U}_{h}\|_{2}<Ch^{k}.
		\end{aligned}
	\end{equation*}
\end{theorem}

\subsection{Corrections for the one- and two-dimensional discretizations}
From the above two theorems, we need to change some properties of the weights produced by the discretization in Sec. 3 for one- and two-dimensional cases.

\subsubsection{One-dimensional case}
Here we provide a lemma to state the properties of weights $w_{i}$ defined in \eqref{eqdefw}.
\begin{theorem}\label{thmfraweightprop1d}
	Let $w_{i}$ be defined in \eqref{eqdefw}. Then $w_{i}$ satisfies:
	\begin{equation*}
		\begin{aligned}
			& w_{i}<0,~~|i|\geq 2; \quad w_{i}=w_{-i},~~i \geq 0;\\
			&\sum_{i=-N+1}^{N-1}w_{i}\geq CL^{-2s},
		\end{aligned}
	\end{equation*}
where $2L$ means the length of $\Omega_1$.
\end{theorem}
\begin{proof}
	
	The definition of $c_{1,s-1}$ and simple calculations give, for $\zeta>h$,
	\begin{equation*}
		\begin{aligned}
			c_{1,s-1}<0,~~ (\zeta-h)^{1-2s}-2\zeta^{1-2s}+(\zeta+h)^{1-2s}<0, ~~{\rm for }~ s<\frac{1}{2} ;\\
			c_{1,s-1}>0,~~ (\zeta-h)^{1-2s}-2\zeta^{1-2s}+(\zeta+h)^{1-2s}>0,~~ {\rm for }~ s>\frac{1}{2},
		\end{aligned}
	\end{equation*}
	which leads to $w_{i}<0$, $|i|\geq 2$.
	As for $w_{1}$, simple calculations give
	\begin{equation}\label{weight1 11}
		w_{1}=-c_{1,s-1}h^{-2s}\frac{7 - 2^{5 - 2s} + 3^{3 -2s}}{(2s-3)(2s-2)}.
	\end{equation}
	Summing $w_{i}$ from $-N+1$ to $N-1$ gives
	\begin{equation*}
		\begin{aligned}
			\sum_{i=-N+1}^{N-1}w_{i}
			= &\frac{1}{h^{2}}\sum_{i=-N+1}^{N-1}(2\bar{\omega}_{i}-\bar{\omega}_{i+1}-\bar{\omega}_{i-1})
			=2\frac{\bar{\omega}_{N-1}-\bar{\omega}_{N}}{h^{2}}.
		\end{aligned}
	\end{equation*}	
	According to the definitions of $\bar{\omega}_{N}$, we have
	\begin{equation*}
		\begin{aligned}
			\frac{\bar{\omega}_{N-1}-\bar{\omega}_{N}}{h^{2}}
			= &\frac{c_{1,s-1}\int_{-h}^{h}(((N-1)h+\zeta)^{1-2s}-(Nh+\zeta)^{1-2s})\bar{\phi}_{1}(\zeta)d\zeta}{h^{2}}\\
			\geq& C\frac{\int_{-h}^{h}L^{-2s}\bar{\phi}_{1}(\zeta)d\zeta}{h}
			\geq CL^{-2s},
		\end{aligned}
	\end{equation*}
	which leads to desired results.
\end{proof}
From Theorem \ref{thmcorstd} and the fact $w_{1}>0$  for some $s\in(0,\frac{1}{2})$ (see \eqref{weight1 11}), we find that the numerical scheme constructed by \eqref{eqdefFLH1D} may not be effective. To make the $w_{i}$ satisfy the condition of Theorem \ref{thmcorstd} and get an effective numerical scheme, we do the modifications for $\bar{\omega}_{0}$, i.e.,
\begin{equation}\label{eqdefw0M}
	\bar{\omega}^{M}_{0}=\left\{
	\begin{aligned}
		&0\qquad {\rm if}~w_{1}\geq 0,\\
		&\bar{\omega}_{0}\qquad {\rm if}~w_{1}<0.
	\end{aligned}
	\right.
\end{equation}
Then we obtain a modified scheme
\begin{equation}\label{modifiedscheme1}
	\begin{aligned}
		(-\Delta)_{h}^{s}u(x)\approx(-\Delta)_{h,M}^{s}u(x_{i})=&\sum_{j=1}^{N-1}w^{M}_{j-i}u_{j}+R_{i},
	\end{aligned}
\end{equation}
where
\begin{equation*}
	\begin{aligned}
		w^{M}_{0}=&-\frac{\bar{\omega}_{-1}-2\bar{\omega}^{M}_{0}+\bar{\omega}_{1}}{h^{2}},\quad w^{M}_{i}=w_{i},\quad |i|\geq 2,\\
		w^{M}_{i}=&-\frac{\bar{\omega}^{M}_{0}-2\bar{\omega}_{1}+\bar{\omega}_{2}}{h^{2}},\quad |i|=1.
	\end{aligned}
\end{equation*}
By the definitions of $w^{M}_{1}$ and $\bar{\omega}_{i}$, it is easy to check that $w^{M}_{1}<0$.\\

Next, we present the truncation error of the modified discretization \eqref{modifiedscheme1}.
\begin{theorem}\label{thmfra1dmd}
	Let $s\in(0,\frac{1}{2})\cup(\frac{1}{2},1)$. $(-\Delta)^{s}$ and $(-\Delta)^{s}_{h,M}$ are defined in \eqref{eqdeflap} and \eqref{modifiedscheme1}, respectively. If $u\in C^{1,\alpha}(\bar{\Omega}^{\delta}_{1})$ with some fixed constant $\delta>4h>0$ and $\alpha\in(\max(0,2s-1),1]$, then we have
	\begin{equation*}
		\|((-\Delta)^{s}-(-\Delta)^{s}_{h,M})u\|_{\infty}\leq Ch^{1+\alpha-2s},\quad \|((-\Delta)^{s}-(-\Delta)^{s}_{h,M})u\|_{2}\leq Ch^{1+\alpha-2s},
	\end{equation*}
	where  $\Omega^{\delta}_{1}=(-L-\delta,L+\delta)$.
\end{theorem}
\begin{proof}
	For fixed $i$, by triangle inequality and Theorem \ref{onetwodimendis}, we obtain
	\begin{equation*}
		\begin{aligned}
			|((-\Delta)^{s}-(-\Delta)^{s}_{h,M})u_{i}|\leq& |((-\Delta)^{s}-(-\Delta)^{s}_{h})u_{i}|+|((-\Delta)^{s}_{h}-(-\Delta)^{s}_{h,M})u_{i}|\\
			\leq& Ch^{1+\alpha-2s}+\vartheta
		\end{aligned}
	\end{equation*}
	As for $\vartheta$, if $\bar{\omega}^{M}_{0}=\bar{\omega}_{0}$, there is
	$\vartheta=0$. Otherwise, we have
	\begin{equation*}
		\begin{aligned}
			\vartheta\leq& \left |(-\Delta)_{h}u_{i}\int_{-h}^{h}|y|^{1-2s}\bar{\phi_{1}}(y)dy\right |
			\leq Ch^{1+\alpha-2s}\|u\|_{C^{1,\alpha}(\bar{\Omega}^{\delta}_{1})},
		\end{aligned}
	\end{equation*}
	which leads to the desired results.
\end{proof}
Thus  we can get the following convergence results for one-dimensional case by Theorem \ref{thmcorstd}.
\begin{theorem}\label{thm1dcon}
	Assume $s\in(0,\frac{1}{2})\cup(\frac{1}{2},1)$. Let $u$ and $\mathbf{U}_{h}$ be solutions of Eqs. \eqref{eqtosol} and \eqref{eqmatrixform} with $b_{i}=w^{M}_{i}$ and
	\begin{equation*}
		\mathbf{U}=\{u(x_{i})\}_{i=1}^{N-1},\quad\mathbf{G}=\left \{(-\Delta)_{h}R_{i}\right \}_{i=1}^{N-1},\quad \mathbf{F}=\{f(x_{i})\}_{i=1}^{N-1}.%-\frac{R_{i-1}+R_{i+1}-2R_{i}}{h^{2}}
	\end{equation*}
	If $u\in C^{1,\alpha}(\bar{\Omega}^{\delta}_{1})$ with some fixed constant $\delta>4h>0$ and $\alpha\in(\max(0,2s-1),1]$, then we have
	
	\begin{equation*}
		\begin{aligned}
			&\|\mathbf{U}-\mathbf{U}_{h}\|_{\infty}<Ch^{1+\alpha-2s},\quad\|\mathbf{U}-\mathbf{U}_{h}\|_{2}<Ch^{1+\alpha-2s}.
		\end{aligned}
	\end{equation*}
	
\end{theorem}

\subsubsection{ Two-dimensional case}
According to Theorem \ref{thmcorstd2}, to obtain an effective numerical scheme, we need to make $w_{i,j}$ satisfy the following requirements
\begin{equation*}
	w_{0,0}^{M}>0 ~~{\rm and}~~w_{i,j}^{M}<0,
\end{equation*}
where $(i,j)\in\{(\pm 1,0),(0,\pm 1),(\pm 1,\pm 1)\}$.

To be specific, we modify the $\bar{\omega}_{i,j}$ as
\begin{equation*}
	\begin{aligned}
		&\bar{\omega}^{M}_{0,0}=\bar{\omega}_{0,0}+c_{0,0},\quad c_{0,0}\geq0;\quad \bar{\omega}^{M}_{i,j}=\bar{\omega}_{i,j},
	\end{aligned}
\end{equation*}
and take $w_{i,j}^{M}$ as
\begin{equation}\label{eqdeffra2dwmd}
	\begin{aligned}
		w_{i,j}^{M}=(\theta(-\Delta)_{h,1}+(1-\theta)(-\Delta)_{h,2})\bar{\omega}_{i,j}^{M}.
		%&-\theta\frac{\bar{\omega}_{i+1,j}^{M}+\bar{\omega}_{i-1,j}^{M}+\bar{\omega}_{i,j+1}^{M}+\bar{\omega}_{i,j-1}^{M}-4\bar{\omega}_{i,j}^{M}}{h^{2}}\\
		%&-(1-\theta)\frac{\bar{\omega}_{i+1,j+1}^{M}+\bar{\omega}_{i-1,j+1}^{M}+\bar{\omega}_{i-1,j-1}^{M}+\bar{\omega}_{i+1,j-1}^{M}-4\bar{\omega}_{i,j}^{M}}{2h^{2}}.
	\end{aligned}
\end{equation}
Thus the two-dimensional discretization scheme can be modified  as
\begin{equation}\label{eqdefFLH2Dmd}
	\begin{aligned}
		(-\Delta)^{s}_{h,M}u_{i,j}=&\sum_{i=1}^{N-1}\sum_{j=1}^{N-1}w
		_{p-i,q-j}^{M}u_{i,j}+(\theta(-\Delta)_{h,1}+(1-\theta)(-\Delta)_{h,2})R_{i,j}.
%		&-\theta\frac{R_{i-1,j}+R_{i+1,j}+R_{i,j-1}+R_{i,j+1}-4R_{i,j}}{h^{2}}\\
%		&-(1-\theta)\frac{R_{i-1,j-1}+R_{i+1,j-1}+R_{i-1,j+1}+R_{i+1,j+1}-4R_{i,j}}{2h^{2}}.
	\end{aligned}	
\end{equation}

Similar to the proofs of Theorems \ref{thmfraweightprop1d} and \ref{thmfra1dmd}, there hold
\begin{theorem}\label{thmfraweightprop2d}
	Let $w_{i,j}^{M}$ be defined in \eqref{eqdeffra2dwmd}. Then %And	$w_{i,j}^{M}$ has following properties
	\begin{equation*}
		\begin{aligned}
			&\sum_{i=-N+1}^{N-1}\sum_{j=-N+1}^{N-1}w_{i,j}^{M}\geq CL^{-2s}.
		\end{aligned}
	\end{equation*}
\end{theorem}
\begin{theorem}
	Let $\Omega_{2}^{\delta}=(-L-\delta,L+\delta)\times(-L-\delta,L+\delta)$ and $s\in(0,\frac{1}{2})\cup(\frac{1}{2},1)$. Suppose $(-\Delta)^{s}$ and $(-\Delta)^{s}_{h,M}$ are defined in \eqref{eqdeflap} and \eqref{eqdefFLH2Dmd}, respectively. If $u\in C^{1,\alpha}(\bar{\Omega}_{2}^{\delta})$ with some fixed constant $\delta>4h>0$ and $\alpha\in(\max(0,2s-1),1]$, then we have
	\begin{equation*}
		\|((-\Delta)^{s}-(-\Delta)^{s}_{h,M})u\|_{\infty}\leq Ch^{1+\alpha-2s},\quad \|((-\Delta)^{s}-(-\Delta)^{s}_{h,M})u\|_{2}\leq Ch^{1+\alpha-2s}.
	\end{equation*}
\end{theorem}

Thus the corresponding convergence results can be obtained by Theorem \ref{thmcorstd2}.

\begin{theorem}\label{thmcon2}
	Let $u$ and $\mathbf{U}_{h}$ be solutions of \eqref{eqtosol} and \eqref{eqmatrixform2} with $t_{i,j}=w_{i,j}^{M}$ and
	\begin{equation*}
		\begin{aligned}
			&\mathbf{U}=\{u(x_{i},y_{j})\}_{i,j=1}^{N-1},\quad \mathbf{F}=\{f(x_{i},y_{j})\}_{i,j=1}^{N-1},\\
			&\mathbf{G}=\left \{(\theta(-\Delta)_{h,1}+(1-\theta)(-\Delta)_{h,2})R_{i,j}\right \}_{i,j=1}^{N-1}.
		\end{aligned}
	\end{equation*}
	After choosing suitable $\theta$ and $c_{0,0}$, we have if $u\in C^{1,\alpha}(\bar{\Omega}^{\delta}_{2})$ with some fixed constant $\delta>4h>0$ and $\alpha\in(\max(0,2s-1),1]$,
	
	\begin{equation*}
		\begin{aligned}
			&\|\mathbf{U}-\mathbf{U}_{h}\|_{\infty}<Ch^{1+\alpha-2s},\quad\|\mathbf{U}-\mathbf{U}_{h}\|_{2}<Ch^{1+\alpha-2s},
		\end{aligned}
	\end{equation*}
where $s\in(\frac{1}{250},\frac{1}{2})\cup (\frac{1}{2},1)$.
\end{theorem}
\begin{remark}
	By numerical experiments, we give the range of $\theta$ with different $s\in(\frac{1}{250},\frac{1}{2})\cup (\frac{1}{2},1)$ and $c_{0,0}$ in Figure \ref{figctheta} (shown in the shaded area), which makes above estimates hold. But for smaller $s$, we do not find a suitable $\theta$ to make $w^{M}_{i,j}$  satisfy Theorem \ref{thmcorstd2}.
	\begin{figure}
		\centering
		\subfigure[$c_{0,0}=1$]{
		\includegraphics[width=0.45\linewidth,angle=0]{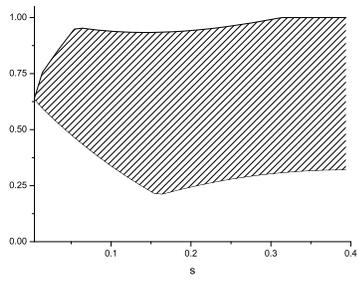}\label{fig:c1}}
		\subfigure[$c_{0,0}=3$]{
		\includegraphics[width=0.45\linewidth,angle=0]{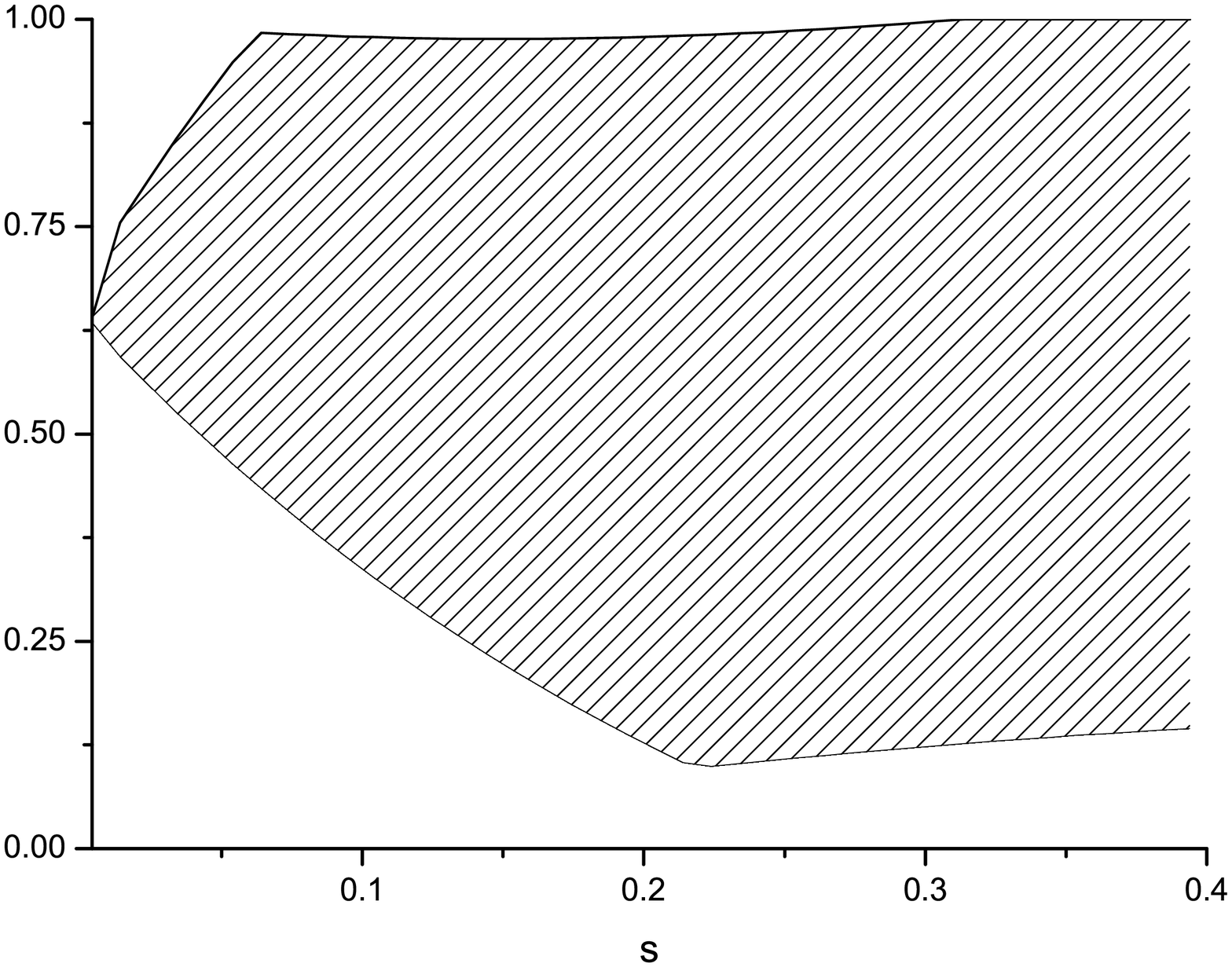}\label{fig:c3}}
	\subfigure[$c_{0,0}=7$]{
		\includegraphics[width=0.45\linewidth,angle=0]{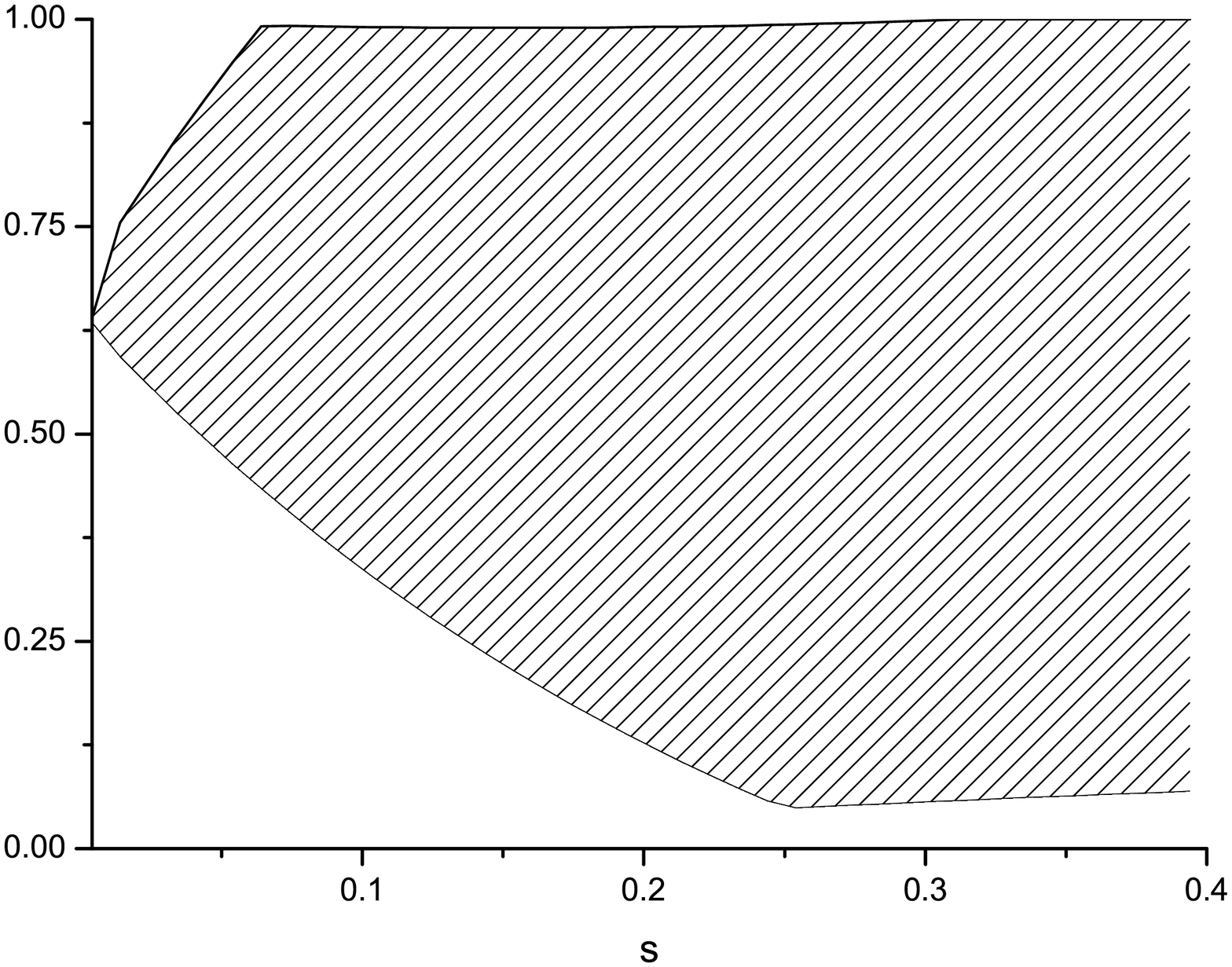}\label{fig:c7}}
		\subfigure[$c_{0,0}=16$]{
		\includegraphics[width=0.45\linewidth,angle=0]{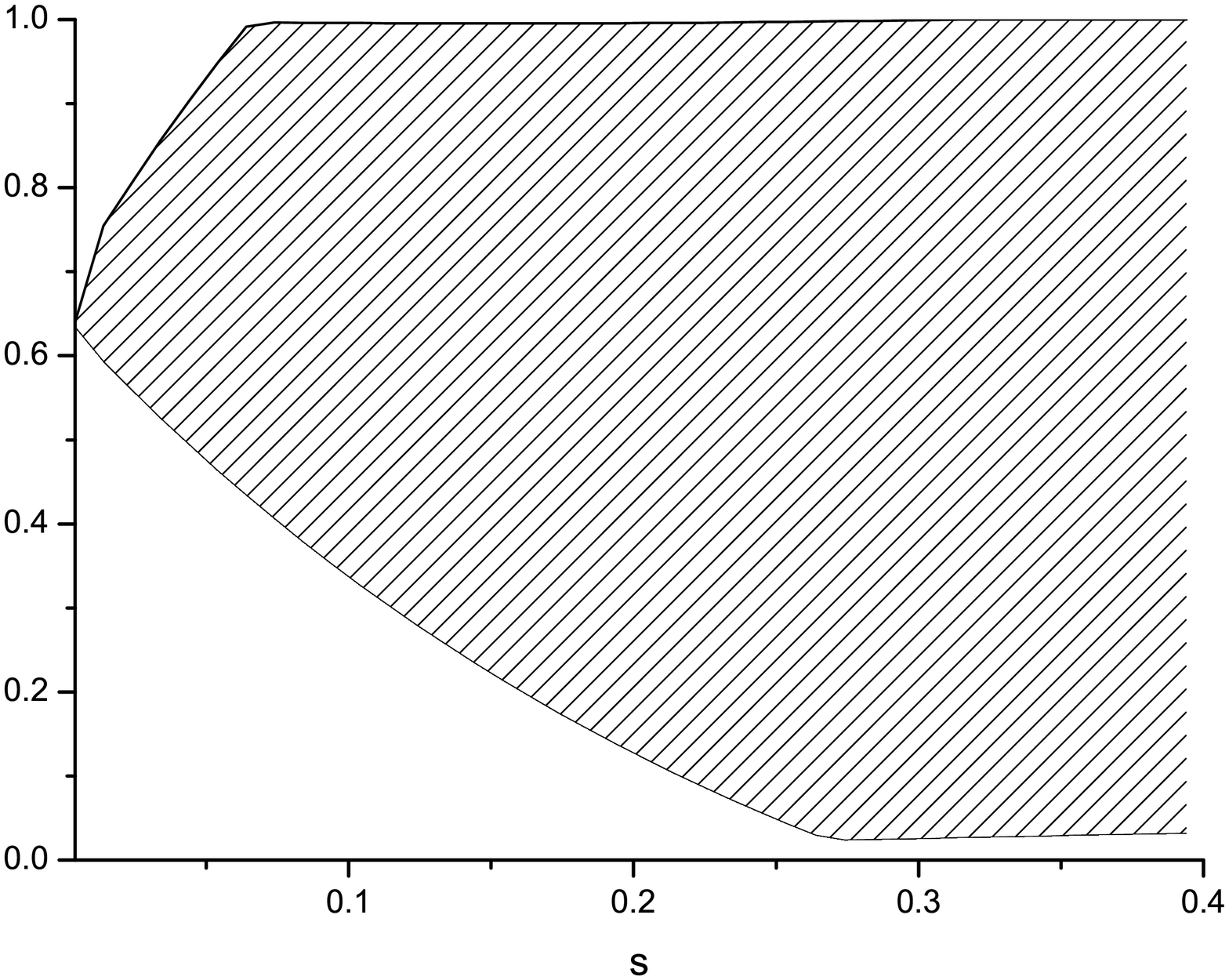}\label{fig:c16}}
		\caption{Range of $\theta$ for different $s$ and $c_{0,0}$.}
		\label{figctheta}
	\end{figure}
	
\end{remark}

\begin{remark}
It is easy to check that the coefficient $c_{n,s}$ in \eqref{eqdeflap} can tend to $\infty$ when $s= \frac{1}{2}$ in one-dimensional case, but it doesn't for the two-dimensional case. %fit to the two-dimensional case.
\end{remark}
\section{Numerical experiments}
In this section, we first verify the convergence of the numerical method in discretizing $(-\Delta)^s$ and solving Eq. \eqref{eqtosol}. Then we simulate the mean exit time of L\'{e}vy motion with generator $\mathcal{A}=\nabla P(x)\cdot\nabla+ (-\Delta)^{s}$.
From \cite{dyda2012}, we have
\begin{equation}\label{eqexactsol1d}
	u=\left\{
	\begin{aligned}
		(1-x^{2})^{P+s},&\quad x\in(-1,1),\\
		0,&\quad otherwise
	\end{aligned}\right.
\end{equation}
with $P\in \mathbb{R}$ and
\begin{equation*}
	(-\Delta)^{s}u=\frac{2^{2s}\Gamma(\frac{1}{2}+s)\Gamma(P+1+s)}{\sqrt{\pi}\Gamma(P+1)}~_{2}F_{1}\left (\frac{1}{2}+s,-P;\frac{1}{2};x^{2}\right ),\quad x\in(-1,1),
\end{equation*}
with $~_{2}F_{1}$ being the Gauss hypergeometric function. Using this result, we test the truncation errors and the convergence rates (the right hand side and boundary terms of Eq. \eqref{eqtosol} are taken as the corresponding expressions).
\begin{example}
In this example, we consider the truncation error in one-dimensional case. Here we choose $\Omega_{1}=(-1,1)$, $g(x)=0$, and $P=2-s$ in \eqref{eqexactsol1d}.  All the results presented in Table \ref{tab:1dtr} agree with Theorem \ref{onetwodimendis}.%The $(-\Delta)^{s}_{h}u$ with $h=\frac{1}{1024}$ can taken to be the reference solution.

\begin{table}[htbp]
	\centering
	\caption{$l^{\infty}(\Omega)$ truncation errors and convergence rates with $P=2-s$}
	\begin{tabular}{ccccc}
		\hline
		$s\backslash 2/h$ & 128&256 & 512 & 1024 \\
		\hline
		0.2 & 2.313E-03&8.321E-04 & 2.930E-04 & 1.016E-04 \\
		& Rates&1.4749 & 1.5061 & 1.5274 \\
		0.4 & 4.168E-03&1.723E-03 & 7.215E-04 & 3.057E-04 \\
		&Rates&1.2742 & 1.2559 & 1.2388 \\
		0.6 & 1.776E-02&9.859E-03 & 5.541E-03 & 3.142E-03 \\
		& Rates&0.8495 & 0.8314 & 0.8185 \\
		0.8 & 6.368E-02&4.724E-02 & 3.536E-02 & 2.662E-02 \\
		& Rates&0.4309 & 0.4179 & 0.4098 \\
		\hline
	\end{tabular}
	\label{tab:1dtr}

\end{table}

\end{example}

\begin{example}
In this example, we use numerical scheme \eqref{eqmatrixform} to solve \eqref{eqtosol} with $g(x)=0$ and $\Omega_{1}=(-1,1)$. Here, we choose $P=1$ in \eqref{eqexactsol1d} which leads to $u\in C^{1,s}(\Omega_{1})$. The results presented in Table \ref{tab:1dP1D} show that the numerical scheme \eqref{eqmatrixform} has an $\mathcal{O}(h^{1+s})$ convergence rate which is higher than the one $(\mathcal{O}(h^{1-s}))$ predicted in Theorem \ref{thm1dcon}.

\begin{table}[htbp]
	\centering
	\caption{$l^{\infty}(\Omega)$ errors and convergence rates with $P=1$}
	\begin{tabular}{cccccc}
		\hline
		$s\backslash 2/h$& 128&256 & 512 & 1024 \\
		\hline
		0.1 & 3.663E-04&1.911E-04 & 9.470E-05 & 4.570E-05 \\
		& Rates&0.9389 & 1.0128 & 1.0512 \\
		0.2 & 1.005E-03&4.965E-04 & 2.329E-04 & 1.061E-04 \\
		& Rates&1.0175 & 1.0922 & 1.1337 \\
		0.3 & 4.274E-04&1.859E-04 & 7.800E-05 & 3.219E-05 \\
		& Rates&1.2011 & 1.2529 & 1.2770 \\
		0.6 & 2.415E-04&7.109E-05 & 2.433E-05 & 8.171E-06 \\
		& Rates&1.7644 & 1.5470 & 1.5741 \\
		\hline
	\end{tabular}
	\label{tab:1dP1D}
\end{table}

\end{example}
\begin{example}
	We choose $P=0$ in \eqref{eqexactsol1d}. We first take $\Omega_{1}=(-0.5,0.5)$ and
	\begin{equation*}
		g(x)=\left\{
		\begin{aligned}
			(1-x^{2})^{P+s},&\quad x\in(-1,1)\backslash \Omega_{1},\\
			0,&\quad otherwise
		\end{aligned}\right.
	\end{equation*}
	 to verify the convergence when we use \eqref{eqmatrixform} to solve the inhomogeneous Dirichlet problem. According to Eq. \eqref{eqdefw0M}, we have $\bar{\omega}_{0}^{M}=0$ when $s=0.2$ and $\bar{\omega}_{0}^{M}=\bar{\omega}_{0}$ when $s=0.3,0.6,0.7$. From the results presented in Table \ref{tab:1dP0nD}, we find when $\bar{\omega}_{0}^{M}=0$, the convergence rates are $\mathcal{O}(h^{2-2s})$ which are the same as the ones predicted by Theorem \ref{thm1dcon} and when $\bar{\omega}_{0}^{M}\neq 0$, the convergence rates are $\mathcal{O}(h^{2})$ which are higher than the predicted ones.

\begin{table}[htbp]
	\centering
	\caption{$l^{\infty}(\Omega)$ errors and convergence rates with $P=0$}
	\begin{tabular}{cccccc}
		\hline
		$s\backslash 2/h$& 128&256 & 512 & 1024 \\
		\hline
		0.2 & 5.20E-05&1.74E-05 & 5.82E-06 & 1.95E-06 \\
		& Rates&1.5763 & 1.5834 & 1.5799 \\
		0.3 & 6.432E-06&1.622E-06 & 4.064E-07 & 1.007E-07 \\
		& Rates&1.9870 & 1.9973 & 2.0123 \\
		0.6 & 6.647E-06&1.749E-06 & 4.547E-07 & 1.167E-07 \\
		& Rates&1.9265 & 1.9434 & 1.9617 \\
		0.7 & 6.474E-06&1.718E-06 & 4.505E-07 & 1.166E-07 \\
		& Rates&1.9143 & 1.9307 & 1.9502 \\
		\hline
	\end{tabular}
	\label{tab:1dP0nD}
\end{table}

Afterwards, we show the numerical results that use \eqref{eqmatrixform} to solve \eqref{eqtosol} with $\Omega_{1}=(-1,1)$ and $g(x)=0$ in Table \ref{tab:1dP0D}. Due to $P=0$, the exact solution has a low regularity. The results presented in Table \ref{tab:1dP0D} show the numerical scheme \eqref{eqmatrixform} is still effective.
\begin{table}[htbp]
	\centering
	\caption{$l^{\infty}(\Omega)$ errors and convergence rates with $P=0$}
	\begin{tabular}{ccccc}
		\hline
		$s\backslash 2/h$& 256&512 & 1024 & 2048 \\
		\hline
		0.2 & 7.681E-02&6.680E-02 & 5.812E-02 & 5.058E-02 \\
		& Rates&0.2016 & 0.2008 & 0.2004 \\
		0.4 & 8.422E-03&6.387E-03 & 4.842E-03 & 3.670E-03 \\
		& Rates&0.3990 & 0.3995 & 0.3997 \\
		0.6 & 2.459E-03&1.621E-03 & 1.069E-03 & 7.053E-04 \\
		& Rates&0.6010 & 0.6005 & 0.6002 \\
		0.8 & 4.966E-04&2.859E-04 & 1.644E-04 & 9.450E-05 \\
		& Rates&0.7964 & 0.7982 & 0.7991 \\
		\hline
	\end{tabular}
	\label{tab:1dP0D}
\end{table}

\end{example}

\begin{example}
	Here we present some examples in two dimensions. We choose
	\begin{equation*}
		u=\left\{
		\begin{aligned}
			((1-x^{2})(1-y^{2}))^{2}, &\quad(x,y)\in \Omega_{2};\\
			0, ~&\quad(x,y) \in \Omega_{2}^{c},
		\end{aligned}\right.
	\end{equation*}
where $\Omega_{2}=(-1,1)\times(-1,1)$ and $g(x,y)=0$. Table \ref{tab:trun2d} shows the truncation errors when using \eqref{eqdefFLH2Dall} with $\theta=0~ {\rm and}~ 1$~to approximate $(-\Delta)^{s}u$. Since $(-\Delta)^{s}u$ is unknown, the truncation errors are calculated by
\begin{equation*}
	e_{h}=\|(-\Delta)^{s}_{h}u-(-\Delta)^{s}_{h/2}u\|_{\infty}.
\end{equation*}
  All the results validate Theorem \ref{onetwodimendis}.
	
	\begin{table}[htbp]
		\centering
		\caption{$l^{\infty}(\Omega)$ truncation errors and convergence rates in two dimensions}
		\begin{tabular}{ccccc}
			\hline
			$(s,\theta)\backslash 2/h$& 64 & 128&512 & 1024 \\
			\hline
			(0.3,0) & 1.238E-03 & 4.994E-04&1.968E-04 & 7.645E-05 \\
			& Rates & 1.3099&1.3437 & 1.3641 \\
			(0.3,1) & 1.324E-03 & 5.208E-04&2.021E-04 & 7.778E-05 \\
			& Rates & 1.3461&1.3656 & 1.3777 \\
			(0.8,0) & 1.358E-01 & 9.929E-02&7.399E-02 & 5.562E-02 \\
			& Rates & 0.4516&0.4244 & 0.4116 \\
			(0.8,1) & 1.332E-01 & 9.868E-02&7.384E-02 & 5.559E-02 \\
			& Rates & 0.4330&0.4183 & 0.4097 \\
			\hline
		\end{tabular}
		\label{tab:trun2d}
	\end{table}
	
In Table \ref{tab:con2d}, we show the convergence of the numerical scheme \eqref{eqmatrixform2}.  Since $(-\Delta)^{s}u$ is unknown, we use $(-\Delta)^{s}_{h}u$ with $h=\frac{1}{2048}$ and $\theta=1$ to approximately represent it. For $s=0.2,0.3$, we take $c_{0,0}=1$ and $\theta=0.5$; the convergence rates presented in Table \ref{tab:con2d} are the same as the ones predicted by Theorem \ref{thmcon2}. For $s=0.4,0.8$, we choose $c_{0,0}=0$ and $\theta=1$; the convergence rates are higher than the predicted ones.
	
	\begin{table}[htbp]
		\centering
		\caption{$l^{\infty}(\Omega)$ errors and convergence rates in two dimensions}
		\begin{tabular}{ccccc}
			\hline
			$s\backslash 2/h$& 64 & 128&256 & 512 \\
			\hline
			0.2 & 6.837E-03 & 2.371E-03&8.040E-04 & 2.654E-04 \\
			& 0 & 1.5281&1.5600 & 1.5993 \\
			0.3 & 7.525E-03 & 3.030E-03&1.179E-03 & 4.419E-04 \\
			& 0 & 1.3125&1.3618 & 1.4157 \\
			0.4 & 1.122E-03 & 2.826E-04&7.286E-05 & 1.834E-05 \\
			& 0 & 1.9886&1.9557 & 1.9901 \\
			0.8 & 1.222E-03 & 3.049E-04&7.550E-05 & 1.837E-05 \\
			& 0 & 2.0030&2.0138 & 2.0393 \\
			\hline
		\end{tabular}
		\label{tab:con2d}
	\end{table}
\end{example}
\begin{example}
	Finally, we use the discretization \eqref{eqdefFLH2Dmd} to simulate the mean exit time $u(\mathbf{x})$  of an orbit starting at $\mathbf{x}$, from a two-dimensional bounded interval $\Omega_{2}$. According to Dynkin formula \cite{Applebaum.2009Lpasc,Sato.1999Lpaidd} of Markov processes, $u(\mathbf{x})$ satisfies \cite{Deng.2017Metaepftapwttplwt,Naeh.1990ADAttEP},
\begin{equation*}
	\left\{
	\begin{aligned}
		\mathcal{A}u(\mathbf{x})&=1,\quad {\rm in}~\Omega_{2},\\
		u(\mathbf{x})&=0,\quad {\rm in}~\Omega_{2}^{c},
	\end{aligned}\right.
\end{equation*}
where
	\begin{equation*}
	\mathcal{A}=\nabla P(\mathbf{x})\cdot\nabla +(-\Delta)^{s},
\end{equation*}
$\nabla$ denotes gradient operator, and $P(x)$ is a given potential.
Here, we take $h=1/64$, $c_{0,0}=100$, $\theta=\frac{1}{2}$, $\Omega_{2}=((-1,1))^{2}$, and $P(\mathbf{x})=\kappa(x_{1}^{2}+x_{2}^{2})$ with $\mathbf{x}=(x_{1},x_{2})$. In Figure \ref{figmeanexitt}, we show the mean exit time when taking $s=0.2,~0.4,~0.6,~0.8$, and $\kappa=0.5$. Comparing Figure \ref{fig:t12} with Figures \ref{fig:t14}, \ref{fig:t16}, \ref{fig:t18},  we find the mean exit time becomes longer and boundary
layer phenomena become weaker as $s$ increases. In Figure \ref{figmeanexitt2}, we show the mean exit time with $s=0.6$ and different $\kappa$. We find that the boundary
layer phenomena become stronger and the mean exit time becomes longer as $\kappa$ increases.

\begin{figure}
	\centering
	\subfigure[$s=0.2$]{
		\includegraphics[width=0.45\linewidth,angle=0]{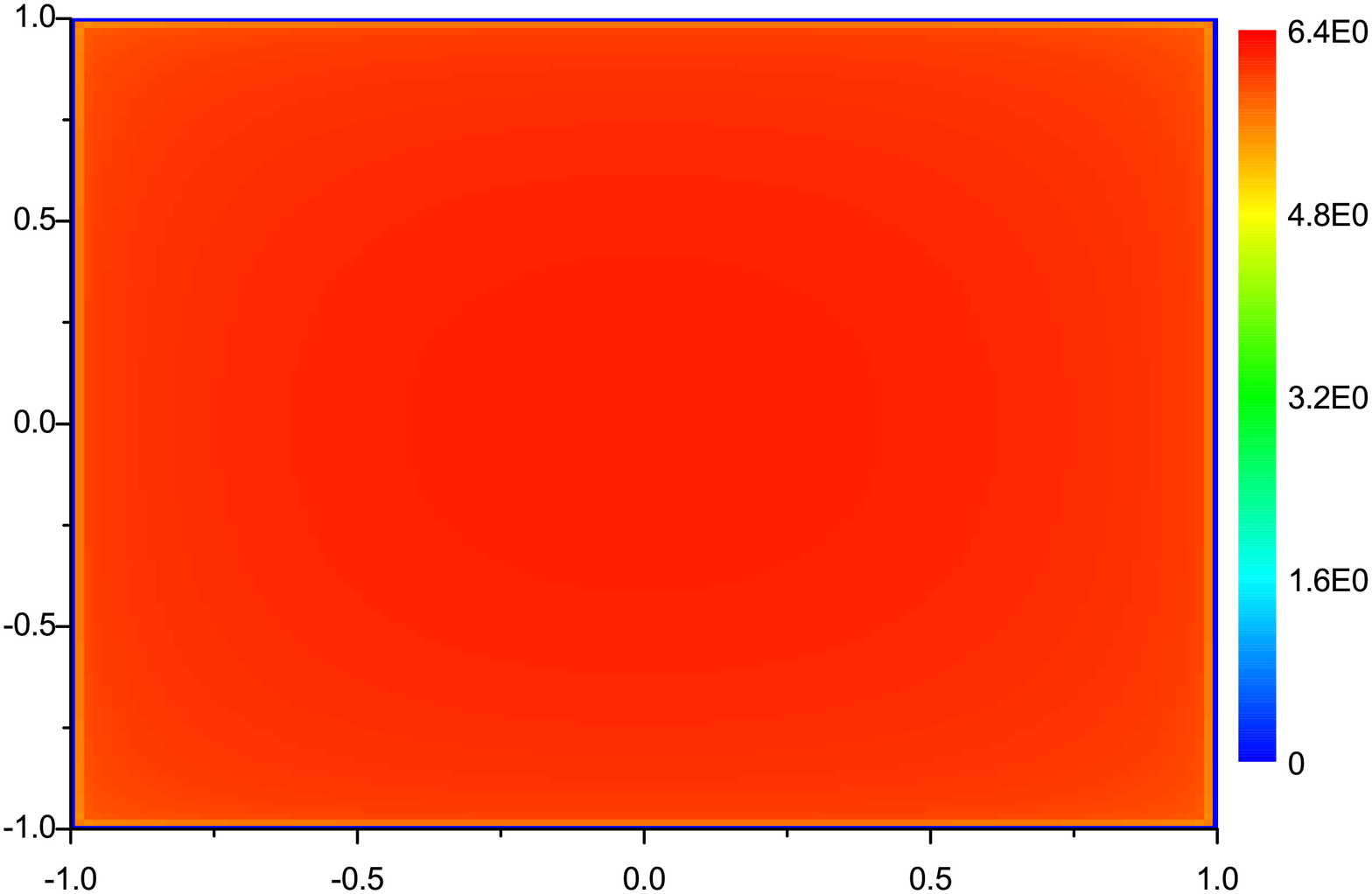}\label{fig:t12}}
	\subfigure[$s=0.4$]{
		\includegraphics[width=0.45\linewidth,angle=0]{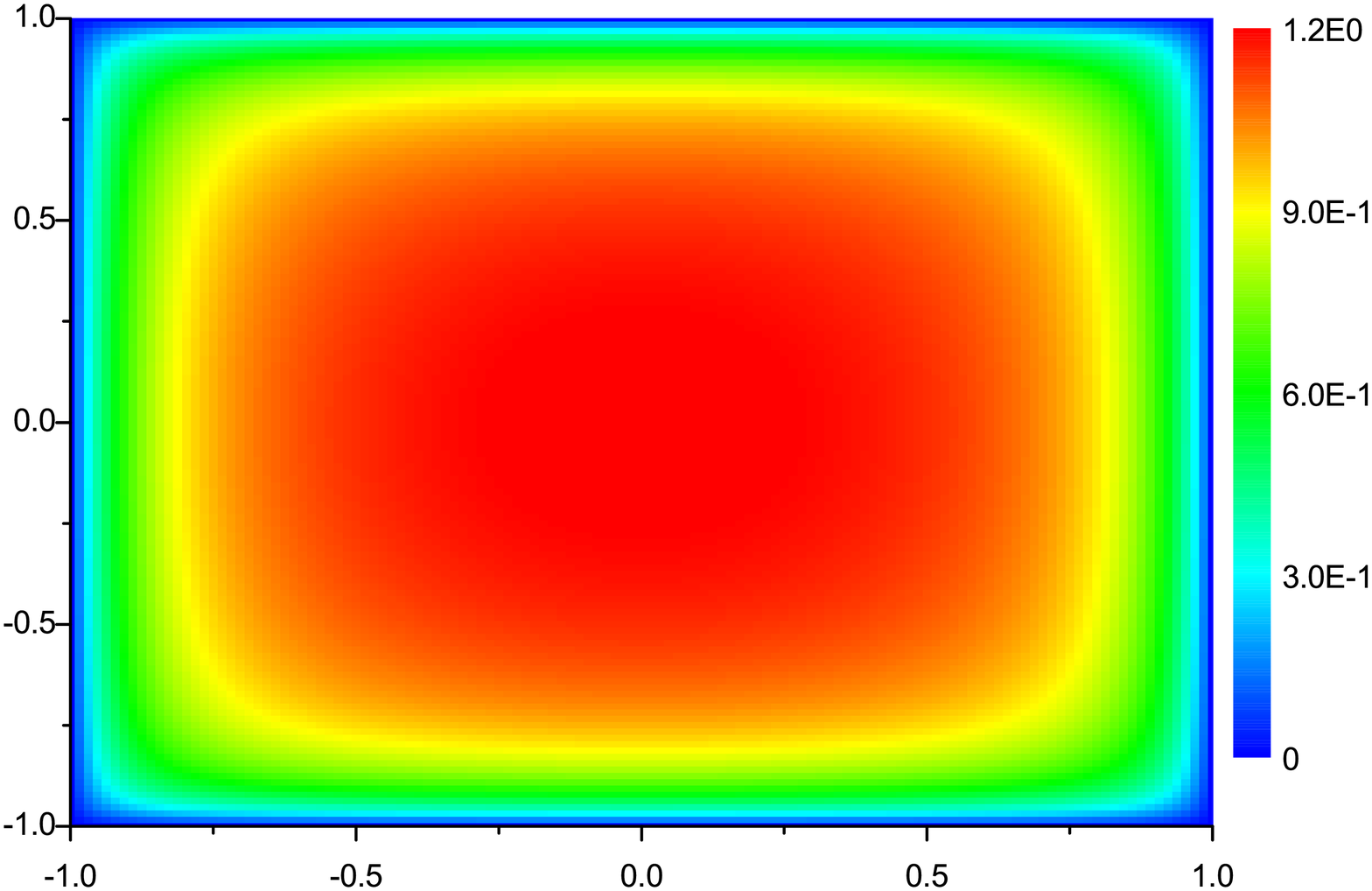}\label{fig:t14}}
	\subfigure[$s=0.6$]{
		\includegraphics[width=0.45\linewidth,angle=0]{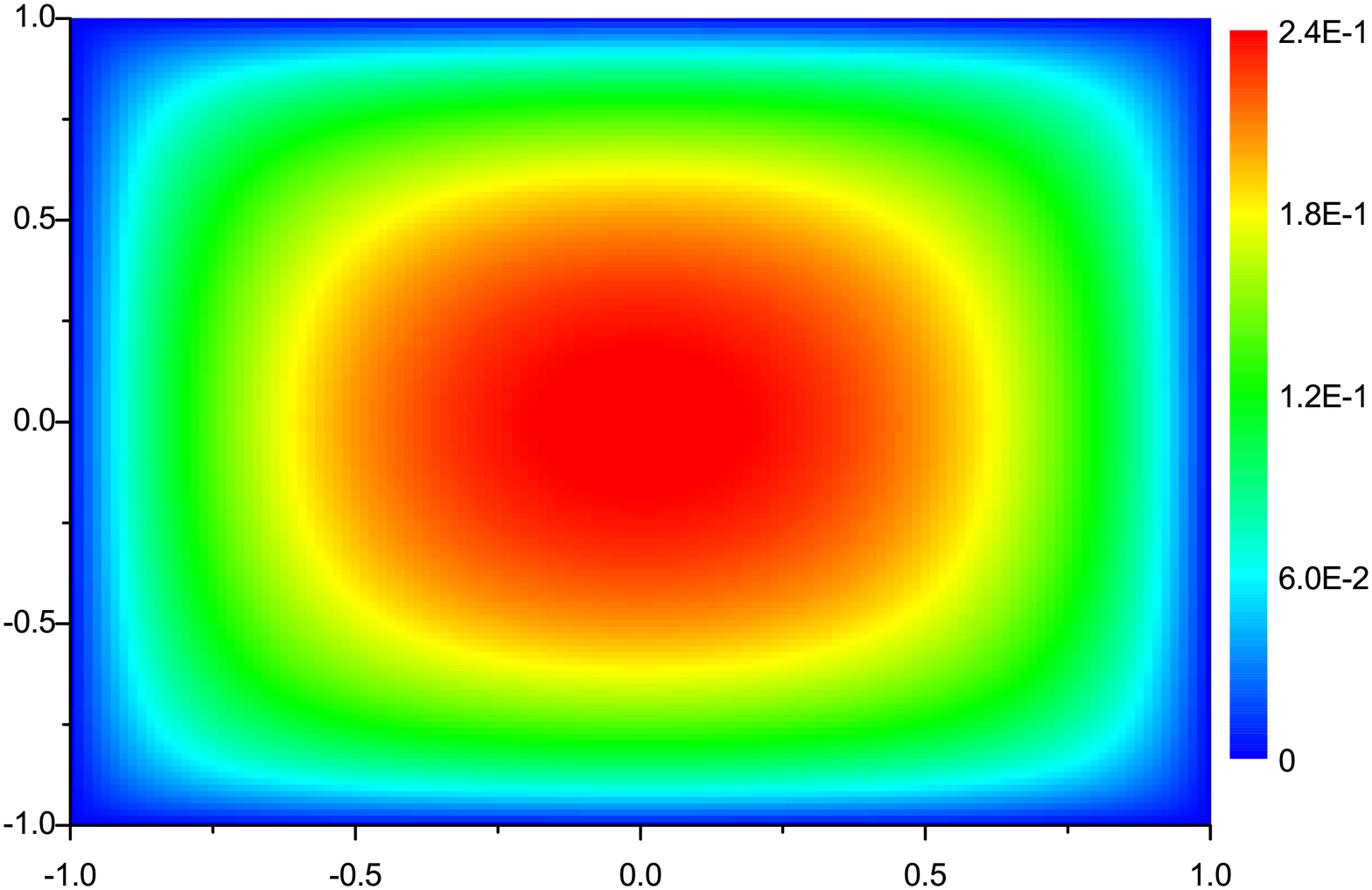}\label{fig:t16}}
	\subfigure[$s=0.8$]{
		\includegraphics[width=0.45\linewidth,angle=0]{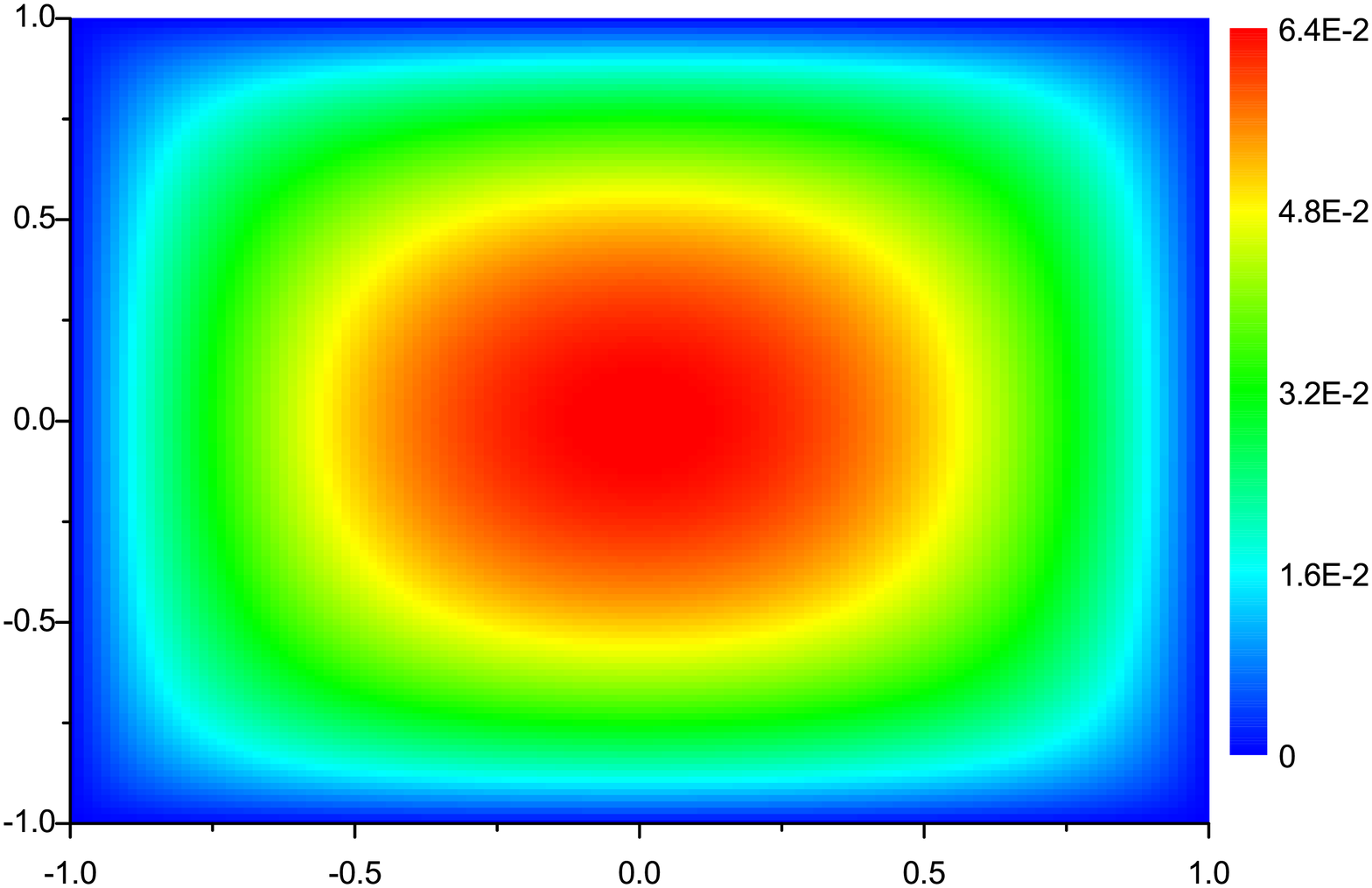}\label{fig:t18}}
	\caption{Mean exit time with $\kappa=0.5$.}
	\label{figmeanexitt}
\end{figure}

\begin{figure}
	\centering
	\subfigure[$\kappa=0.25$]{
		\includegraphics[width=0.45\linewidth,angle=0]{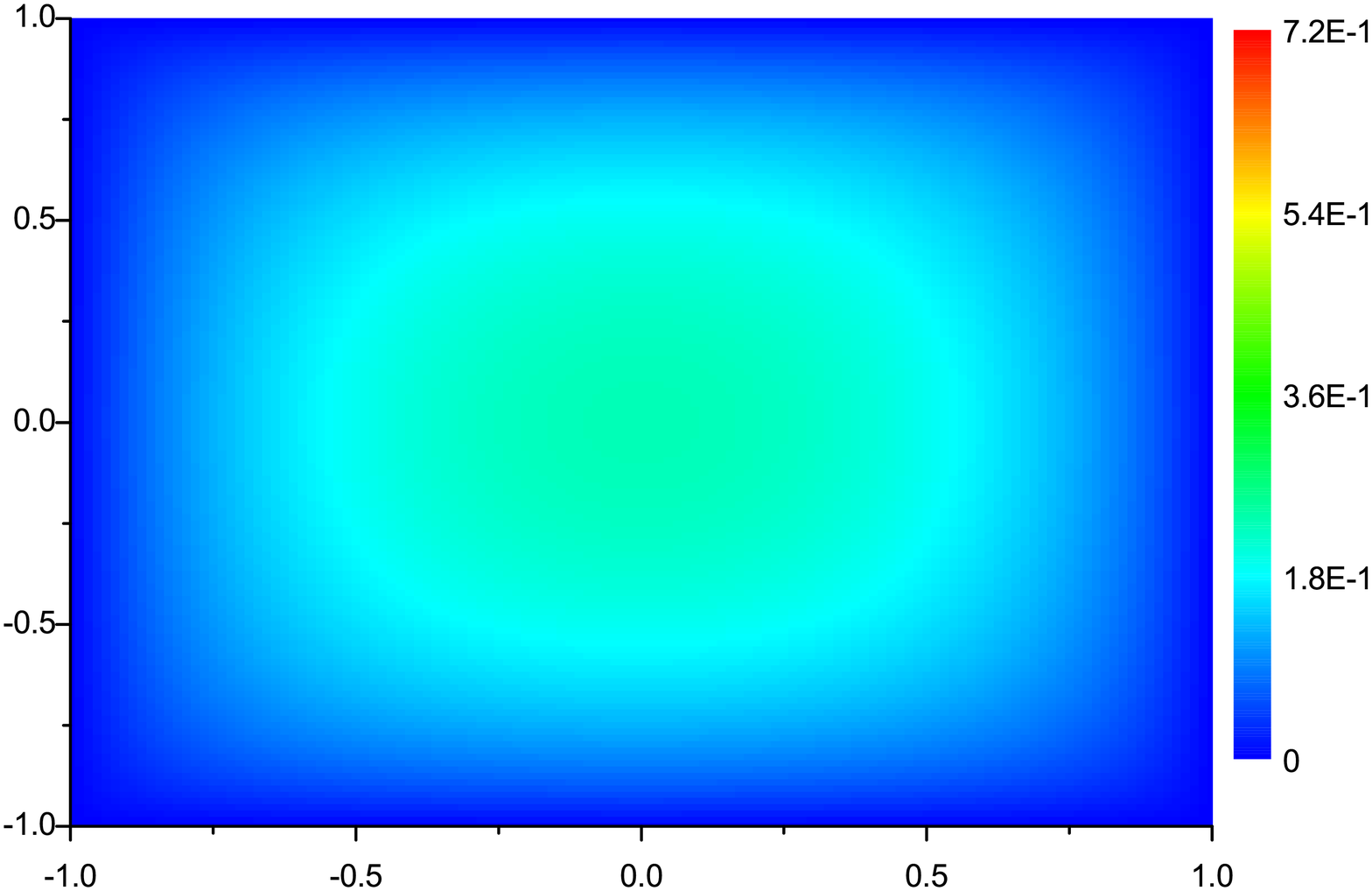}\label{fig:t25}}
	\subfigure[$\kappa=1$]{
		\includegraphics[width=0.45\linewidth,angle=0]{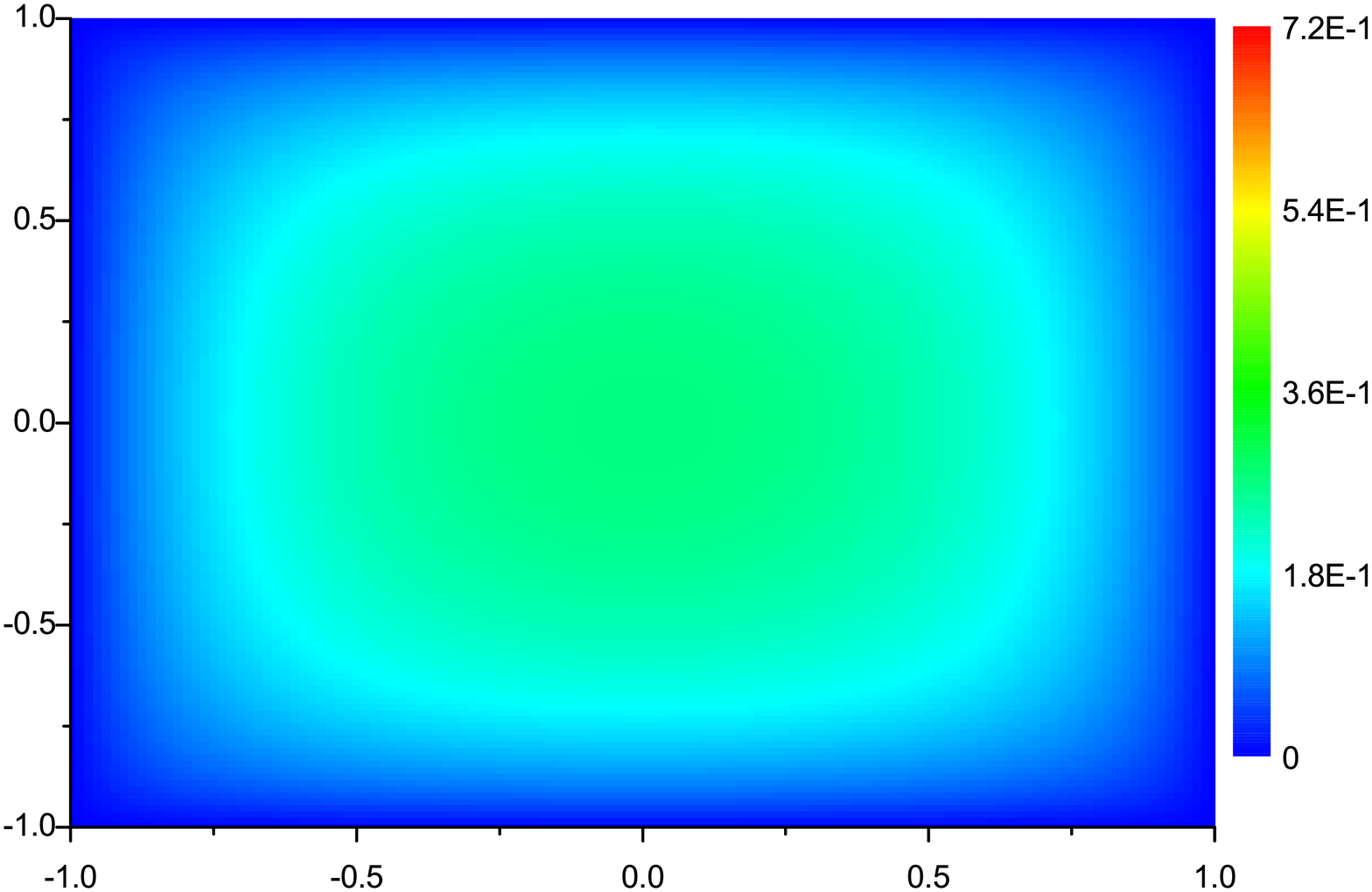}\label{fig:t100}}
	\subfigure[$\kappa=4$]{
		\includegraphics[width=0.45\linewidth,angle=0]{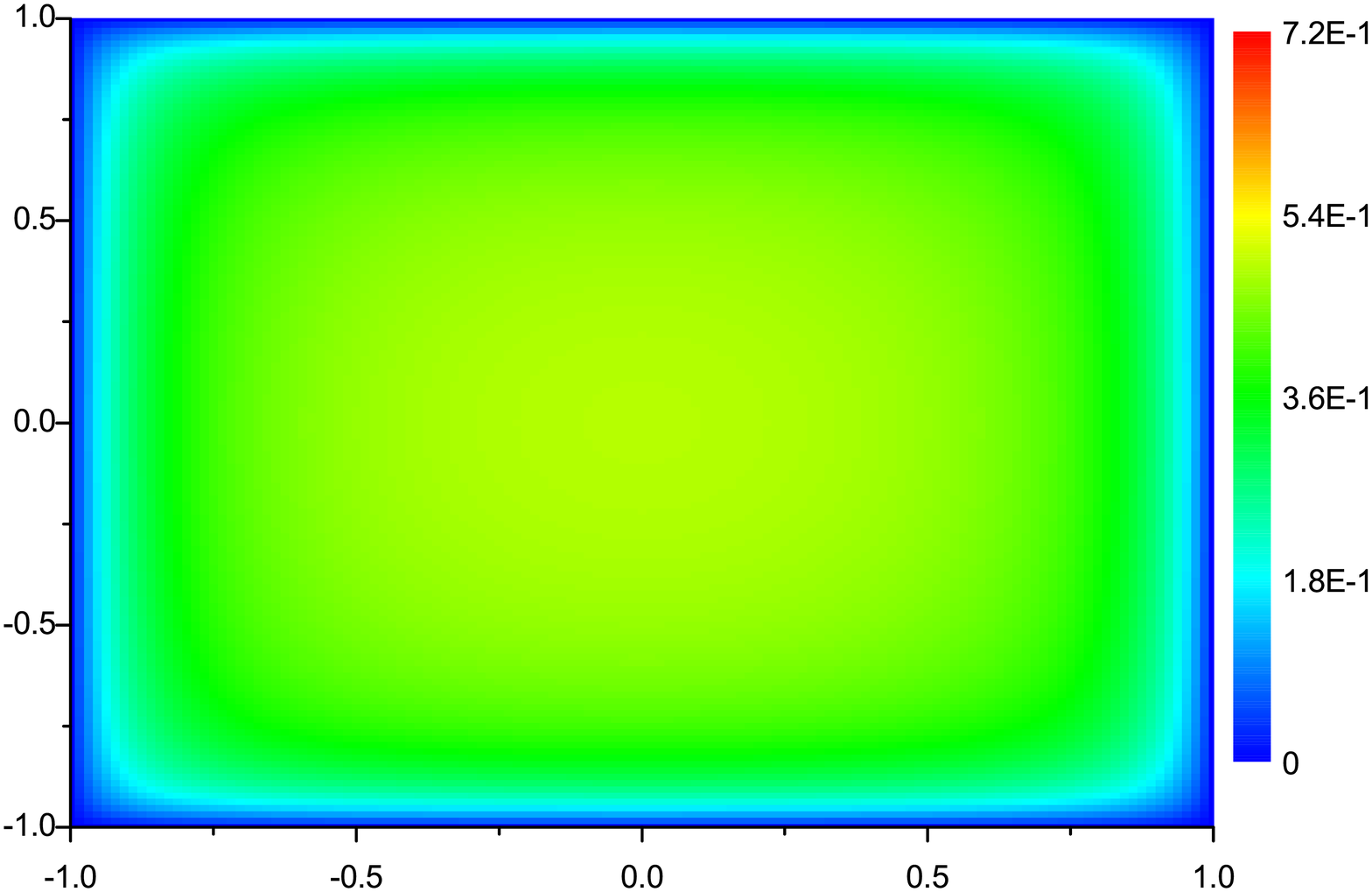}\label{fig:t400}}
	\subfigure[$\kappa=8$]{
		\includegraphics[width=0.45\linewidth,angle=0]{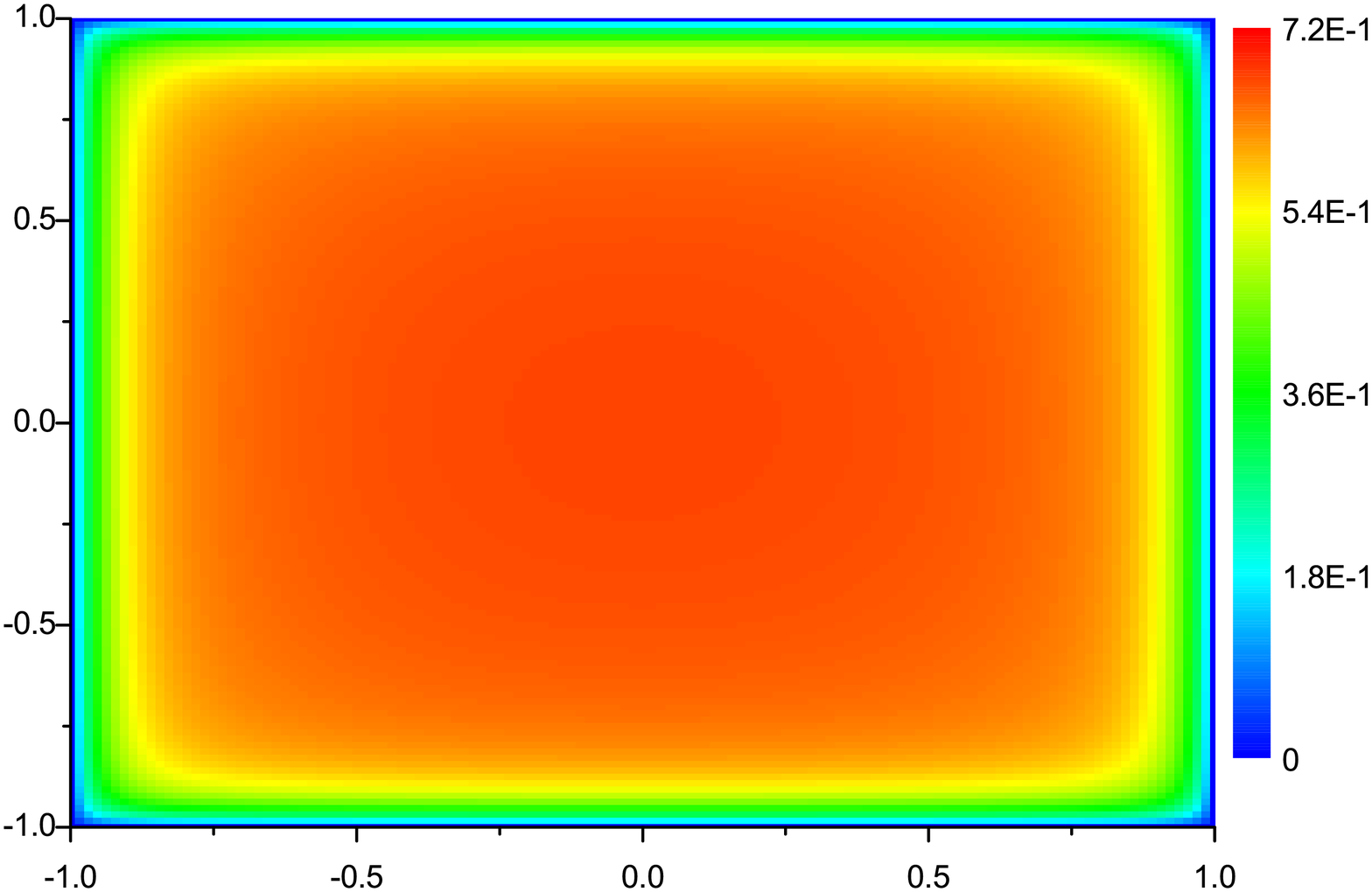}\label{fig:t800}}
	\caption{Mean exit time with $P(\mathbf{x})=\kappa(x_{1}^{2}+x_{2}^{2})$ and $s=0.6$.}
	\label{figmeanexitt2}
\end{figure}
\end{example}
\section{Conclusions}
A fundamentally new idea of discretizing the fractional Laplacian is introduced and used to solve the inhomogeneous fractional Dirichlet problem. The effectiveness of the designed scheme is ensured by the completely theoretical analyses and verified by numerical experiments. Specific applications for simulating the mean exit time of L\'evy processes under harmonic potential are provided; the effects of the strengthes of the potential and the L\'evy exponents are uncovered.

%
%In this paper, we provide a new decomposition for integral fractional Laplacian based on its Fourier transform. And then, we use the Lagrange interpolation and the finite difference methods to discretize the one- and two-dimensional fractional Laplacian. Moreover, we use our discretization to solve the inhomogeneous fractional Dirichlet problem and get the convergence by correcting our discretization schemes. Finally, the numerical experiments validate the effectiveness of our algorithm.

\section*{Acknowledgements}
This work was supported by the National Natural Science Foundation of China under Grant No. 12071195, and the AI and Big Data Funds under Grant No. 2019620005000775.

\bibliographystyle{elsarticle-num}

% Loading bibliography database
\bibliography{cas-refs}

%% The Appendices part is started with the command \appendix;
%% appendix sections are then done as normal sections
%% \appendix

%% \section{}
%% \label{}

%% If you have bibdatabase file and want bibtex to generate the
%% bibitems, please use
%%
%%  \bibliographystyle{elsarticle-num}
%%  \bibliography{<your bibdatabase>}

%% else use the following coding to input the bibitems directly in the
%% TeX file.

\end{document}